\documentclass[11pt]{article}
\usepackage{amsthm}
\usepackage[english]{babel}
\usepackage{etoolbox}
\AtBeginEnvironment{thebibliography}{%

  \setlength{\parskip}{0pt}
  \setlength{\baselineskip}{10pt}
    \setlength{\itemsep}{0pt}
}
\usepackage[letterpaper,top=2cm,bottom=2cm,left=3cm,right=3cm,marginparwidth=1.75cm]{geometry}

\usepackage{xcolor}
\definecolor{darkblue}{RGB}{0, 0, 100}
\definecolor{darkorange}{RGB}{200, 100, 0}

\newcommand{\inD}{\in A(D)}

\newcommand{\vP}{\mathcal{P}}
\newcommand{\vQ}{\mathcal{Q}}

\newcommand{\rev}[1]{\overleftarrow{#1}}

\newtheorem{theorem}{Theorem}[section]
\newtheorem{lemma}{Lemma}

\newtheorem{definition}{Definition}

\newtheorem{proposition}{Proposition}
\newtheorem{problem}{Problem}
\newtheorem{conjecture}{Conjecture}
\newtheorem{claim}{Claim}

\usepackage{amsmath}
\usepackage{graphicx}
\usepackage[colorlinks=true, allcolors=blue]{hyperref}
\usepackage{tikz}
\usetikzlibrary{arrows.meta, shapes.geometric}
\usetikzlibrary{positioning, arrows, shapes, calc}
\usepackage{graphicx} 
\usepackage{subcaption}
\usepackage{float} 
\usepackage{amssymb}
\usepackage{amsmath}
\usepackage{enumitem} 
\usepackage{mathtools}

\usepackage{bm} 
\usepackage{extarrows}
\usepackage{mathrsfs} 

\usepackage{lipsum}

\title{Sparse spanning  $k$-strong oriented subdigraphs in  split digraphs\thanks{The author's work is supported by National Natural Science Foundation of China (No.12571373)}}
\author{Jia Zhou\textsuperscript{1}, J{\o}rgen Bang-Jensen\textsuperscript{2,3}, Jin Yan\textsuperscript{2} \footnote{Corresponding author. E-mail address: yanj@sdu.edu.cn.} \unskip\\[2mm]
\textsuperscript{1} \small School of Mathematics and Statistics, Ningxia University, Yinchuan 750021, China\\
\textsuperscript{2} \small School of Mathematics, Shandong University, Jinan 250100, China
\\
\textsuperscript{3} \small Department of Mathematics and Computer Science, University of\\ \small Southern Denmark, Odense DK-5230, Denmark}
\date{ }

\begin{document}
\maketitle
\bibliographystyle{plain}

\begin{abstract}
Jackson and Thomassen conjectured that every $2k$-strong digraph admits a spanning $k$-strong oriented subdigraph [Ann. N. Y. Acad. Sci. 555 (1989) 402–412]. The conjecture holds for $k=1$ but other than some partial results that have been obtained for general $k$ in some special families of digraphs, including symmetric digraphs, the conjecture remains wide open in general. Even the existence of an integer $K$ such that every $K$-strong digraph has a 2-strong spanning oriented subdigraph is open. As a natural optimization counterpart, the minimum spanning $k$-strong subdigraph (MSSS$_k$) problem, is to find the minimum number of arcs in a spanning $k$-strong subdigraph of a $k$-strong digraph. {This problem is NP-hard already for $k=1$ as it generalizes the hamiltonian cycle problem.} In this paper, we address both problems simultaneously for the class of split digraphs, by constructing sparse spanning $k$-strong oriented subdigraphs. Specifically, we prove that every $k$-strong split digraph $D = (V_1, V_2; A)$ with  minimum semi-degree {$\delta^0(D)\geq 26k+15$} contains a spanning $k$-strong oriented subdigraph with no more than $kn+k|V_1|+98k^2+38k+3$  arcs, where $kn + k|V_1|$ is tight and the $k^{2}$ term is tight up to a constant factor.
For the class of $k$-strong tournaments with minimum semi-degree at least { $26k+15$} our results improve the bound obtained by Kang in [Combin. Probab. Comput., 27:892–907, 2018].

\end{abstract}

 \vspace{1ex}
\vspace{1ex}



\section{Introduction}
An {\bf orientation} of a digraph $D=(V,A)$ is any spanning subdigraph $D'=(V,A')$ of $D$ that can be obtained by deleting one arc from every 2-cycle of $D$.
Let \( k \) be a positive integer. A digraph is \textbf{strongly \( \boldsymbol{k} \)-connected} (briefly, $k$-strong) if it has at least \( k+1 \) vertices and, after deleting any set \( S \) of at most \( k-1 \) vertices, there exists a path from \( x \) to \( y \) for any pair of distinct vertices \( x \) and \( y \).  Every $k$-strong digraph trivially admits a spanning $k$-strong subdigraph, while finding a sparse one remains a challenging task already for $k=1$ as it is NP-complete to find the minimum number of arcs in a spanning strong subdigraph of a strong digraph. This paper considers sparsity from two different aspects: one is to eliminate all $2$-cycles to construct a spanning $k$-strong oriented subdigraph, and the other is to minimize the total number of arcs of the desired spanning $k$-strong oriented subdigraph.

The existence of spanning  $k$-strong oriented subdigraphs  {in highly connected digraphs} is a long-standing open problem.
 {Let $\boldsymbol{f_{\mathcal C}(k)}$ denote the smallest integer $r$ such that every $r$-strong digraph in class $\mathcal C$ admits a spanning $k$-strong oriented subdigraph,
with $f_{\mathcal C}(k)=\infty$ if no such integer exists.}  As a generalization of Robbins' theorem, which states that a graph has a strong orientation if and only if it is $2$-edge-connected~\cite{Robbins1939}, Boesch and Tindell~\cite{boesch1980} proved that every $2$-strong digraph contains a spanning strong oriented subdigraph. Jackson and Thomassen \cite{thomassen1989}   {conjectured the following for general integer $k$}.
\begin{conjecture}\label{conj1}
\emph{(Jackson and Thomassen \cite{thomassen1989})} Every $2k$-strong digraph contains a spanning $k$-strong oriented subdigraph.
\end{conjecture}
The lower bound is given by the complete digraph on $2k$ vertices, which is $(2k-1)$-strong  {but contains no spanning $k$-strong oriented subdigraph.} The conjecture remains open even for $k=2$, and it is still unknown whether $f_{\mathcal D}(2)$ is finite, where $\mathcal D$ is the class of all digraphs.

 {Given the difficulty of Conjecture \ref{conj1}, attention has turned to special families of digraphs.} Let $\mathcal S$ denote the class of \textbf{symmetric digraphs}, where $uv$ is an arc if and only if $vu$ is an arc. Jord\'an~\cite{115} obtained a result implying $f_{\mathcal S}(2)\le18$, and Thomassen~\cite{186} later proved that $f_{\mathcal S}(2)=4$. Very recently, Garamv\"olgyi, Jord\'an, Kir\'aly and Vill\'anyi~\cite{garamvolgyi2025highly} proved that $f_{\mathcal S}(k)\le320k^2$. Let $\mathcal{SC}$ denote the class of \textbf{semicomplete digraphs}, that is, digraphs where there is at least one arc between
every pair of distinct vertices. A \textbf{tournament} is a semicomplete digraph with no 2-cycles. Bang-Jensen and Jord\'an~\cite{bangDM310} conjectured that every $(2k-1)$-strong semicomplete digraph on at least $2k+1$ vertices contains a spanning $k$-strong tournament. They proved the case $k=2$, and the case $k=3$ was verified by Wang, Qi and Yan \cite{wang2024spanning}. Let $\mathcal{LSC}$ denote the class of \textbf{locally semicomplete digraphs}, where both the in-neighbourhood and the out-neighbourhood of every vertex induce semicomplete digraphs. Guo~\cite{5} proved  $f_{\mathcal{LSC}}(k)\le3k-2$.  {Let $\mathcal{ESC}$ denote the class of \textbf{extended semicomplete digraphs}, obtained from semicomplete digraphs by replacing each vertex with an independent set. Recently, Zhou, Bang-Jensen, Zhou and Yan~\cite{ZhouBangJensenZhouYan2026} proved that
$f_{\mathcal{ESC}}(k)\le 4k+1$. For other generalizations of semicomplete digraphs, it was shown in \cite{13,61,62} that strong connectivity is sufficient to guarantee the existence of a spanning strong oriented subdigraph.}



The Minimum Spanning Strongly Connected Subdigraph (MSSS) problem asks for the minimum number of arcs in a spanning strong subdigraph of a digraph. This is also a fundamental combinatorial optimization problem for digraphs, as it generalizes the Hamiltonian cycle problem.  {A natural generalization} is the MSSS$_{k}$ problem, which is to find the minimum number of arcs in a spanning $k$-strong subdigraph of a highly strong digraph. Mader \cite{mader1985} first proved that every $k$-strong digraph of order $n\geq 4k+3$ contains a spanning $k$-strong subdigraph with at most $2kn-2k^2$ arcs. {While $2kn-2k^{2}$ is tight for general digraphs, it can be improved in the case of dense digraphs.} In 2009, Bang-Jensen \cite{bang2009} posed the problem of whether there exists a function $g(k)$ such that every $k$-strong tournament has a spanning $k$-strong subdigraph with at most $kn+g(k)$ arcs. {An earlier} result of Bang-Jensen, Huang and Yeo \cite{bangJGT46} implies that $g(k)\geq k(k-1)/2$. {Kang, Kim, Kim and Suh \cite{kang2017} proved the existemce of $g(k)$ by showing }that one can take $g(k)=750k^2\log_2(k+1)$.  In 2019, Kang \cite{kang2018} further improved the upper bound of $g(k)$, which was given in \cite{kang2017}, to {$g(k)\leq 800k^2$}. 
He also established relevant results for dense digraphs and posed the following conjecture. Here, $\delta(D)$ denotes the minimum, over all vertices, of the sum of the numbers of its in-neighbours and out-neighbours, while the minimum semi-degree, $\delta^0(D)$, denotes the minimum over all in-degrees and out-degrees of vertices of $D$.

\begin{conjecture}\label{conj2} \emph{\cite{kang2018}} There is a constant $C > 0$ such that for integers $k, l, n \ge 1$, every $k$-strong digraph $D$ of order $n$ with $\delta(D) \ge n-l$ contains a spanning $k$-strong subdigraph with at most $kn + kl + Ck^2$ arcs.
\end{conjecture}

In this paper, we consider split digraphs which form a natural  generalization of semicomplete digraphs. A \textbf{split digraph} \(D=(V_1,V_2;A)\) is a digraph whose vertex set can be partitioned into an independent set \(V_1\) and a subset \(V_2\) that induces a semicomplete subdigraph.  Due to the lack stucture on the arcs between the sets $V_1$ and $V_2$ some problems that are easy for semicomplete digraphs become very hard on split digraphs. One such example is the hamiltonian cycle problem which is 'trivial' for semicomplete digraphs (every strong semicomplete digraph has a hamiltonian cycle) but NP-complete for general split digraphs. Hence split digraphs, which are relatively 'understudied' form a very interesting class to study. Among the few available structural results, we mention the following. Bang-Jensen and Wang~\cite{BangJensenWangSplit} investigated strong arc decompositions of split digraphs. {Chen, Bang-Jensen, Yan, and Zhou \cite{ChenBangJensenYanZhou2026} studied the $2$-linkage problem for split digraphs.} For the restricted subclass in which $D[V_2]$ is a tournament, Nguyen, Scott and Seymour~\cite{NguyenScottSeymour} recently studied $2$-kernels and verified the Erd\H{o}s--Sz\'ekely conjecture for this class. 
   Our main result addresses the two sparsity aspects above simultaneously, which supports Conjectures \ref{conj1} and \ref{conj2} for the case of split digraphs.

\begin{theorem}\label{main1}
Every $k$-strong split digraph $D=(V_1,V_2;A)$ of order $n$ with { $\delta^0(D)\ge 26k+15$} contains a spanning $k$-strong oriented subdigraph with at most $kn+k|V_1|+98k^2+38k+3$ arcs, which can be constructed in polynomial time.
\end{theorem}

Since our target object is a spanning $k$-strong subdigraph, the $k$-strong connectivity assumption in Theorem~\ref{main1} is best possible. Moreover,  $kn+k|V_1|$ in the arc bound is best possible (see Proposition \ref{prop2} in Section \ref{sec:remarks}), while the quadratic dependence on $k$ is unavoidable due to the result of Bang-Jensen, Huang and Yeo
mentioned above. Thus Theorem~\ref{main1} simultaneously gives a positive result towards the spanning oriented subdigraph problem and a sparse spanning $k$-strong subdigraph result for split digraphs. Under the additional minimum semi-degree condition, Theorem \ref{main1} improves the bound established by Kang \cite{kang2018} in terms of the number of arcs.

The proof of Theorem \ref{main1} uses the sparse linkage framework from~\cite{kang2017} as a foundational tool, yet the key novelties are twofold: we utilise a nearly dominating set technique from our recent work~\cite{zhouDM26}, which is critical for constructing $k$-strong subdigraphs using as few arcs as possible, and we develop a new fan construction to precisely handle the arc overhead while preserving $k$-strong connectivity. For relevant research advances concerning the existence of spanning $k$-arc-strong subdigraphs, we refer the reader to \cite{bangJGT46,berg2005,dalmazzo1977,kang2018, Zhouarxiv}.

The rest of the paper is organized as follows. Section 2 provides notation. Preliminaries and auxiliary lemmas are collected and established in Subsection 3.1, and Theorem \ref{main1} is proved in Subsection 3.2. Section 4 demonstrates the optimality of the arc bound and the necessity of the minimum semi-degree assumption in Theorem \ref{main1}.

\section{Notation}

Notation not specified in this section is consistent with that in \cite{book}.  For an integer $i$, we use the notation  $\boldsymbol{[i]} = \{1, \ldots, i\}$, and $\boldsymbol{[i, i+j]} = \{i, i+1, \ldots, i+j\}$.  Given a digraph \(D = (V, A)\) with vertex set \(V\) and arc set \(A\), the order of \(D\) is denoted by \(\boldsymbol{|D|}\). All digraphs considered herein are simple, i.e., without loops and multiple arcs. For an arc $(u,v)$, we write $uv$ and say $\boldsymbol{u}$ {\bf dominates} $\boldsymbol{v}$, denoted $\boldsymbol{u \to v}$.
For a vertex $x\in V(D)$,
\[
\boldsymbol{N_D^+(x)} = \{y : xy \in A(D)\}, \quad \boldsymbol{d_D^+(x)} = |N_D^+(x)|,\]
\[{\boldsymbol{N_D(x) }= \{y : xy \in A(D)\text{ or } yx\in A(D)\},  \quad \boldsymbol{d_D(x)} = |N_D(x)|},
\]
 denote the {\bf out-neighborhood}, {\bf out-degree}, {\bf neighborhood} and  {\bf degree} of $x$; the {\bf in-neighborhood} $N_D^-(x)$ and {\bf in-degree} $d_D^-(x)$ are defined analogously. For $v\in V(D)$ and $U\subseteq V(D)$, let
$\boldsymbol{d_D^+(v,U)}:=|N_D^+(v)\cap U|$ and
$\boldsymbol{d_D^-(v,U)}:=|N_D^-(v)\cap U|$. {Define $\boldsymbol{\delta^0(D)}=\min \{\min\{d_D^+(x):x\in V(D)\},\min\{d_D^-(x):x\in V(D)\}\}$ and $\boldsymbol{\delta(D)}=\min \{d_D(x):x\in V(D)\} $ by the {\bf minimum semidegree} and the {\bf minimum degree} of $D$.} In this paper, $\boldsymbol{\Delta^0(D)}$ denotes the maximum between the maximum out-degree and the maximum in-degree of $D$.  
For $X\subseteq V(D)$, $\boldsymbol{D\setminus X}$ is the digraph obtained by deleting $X$ and all incident arcs. Given an arc set $E$ containing no 2-cycle, we define the reverse arc set
$\boldsymbol{\rev{E}}=\{uv\inD\mid vu\in E\}. $  Hence $\rev{E}$ contains one arc for each 2-cycle in $D$ which has one of its arcs in $E$, namely the opposite of that arc. 

Let \(P = x_1x_2\cdots x_t\) be a path. Denote $x_1$ (resp. $x_t$) to be the \textbf{initial} (resp. \textbf{terminal}) vertex of $P$ and that $P$ is an $(x_1,x_t)$-path. The \textbf{length} of $P$ is the number of arcs, and we denote a path of length $l$ as an \textbf{$\boldsymbol{l}$-path}. An \((x,y)\)-path \(P\) is {\bf minimal} if no shorter \((x,y)\)-path exists in the induced subdigraph \(D\langle V(P)\rangle\). Let \(Q = y_1y_2\cdots y_t\) be another path. The paths \(P\) and \(Q\) are {\bf disjoint} if \(V(P)\cap V(Q)=\emptyset\), and {\bf internally disjoint} if
\[
\{x_2,\dots,x_{t-1}\}\cap V(Q) = \emptyset \quad \text{and} \quad V(P)\cap \{y_2,\dots,y_{t-1}\}=\emptyset.
\]
And if \(y_1 = x_t\), then \(\boldsymbol{P \circ Q}\) is a walk from \(x_1\) to \(y_t\). Given a collection of paths $\mathcal{Q}$, we use $\boldsymbol{\text{\textbf{Init}} (\mathcal{Q})}$ to represent the set of all initial vertices of paths in $\mathcal{Q}$, $\boldsymbol{\text{\textbf{Ter}} (\mathcal{Q})}$ to represent the set of all terminal vertices of paths in $\mathcal{Q}$, and $\boldsymbol{\text{\textbf{Int}}(\mathcal{Q})} = V(\mathcal{Q}) \setminus (\text{Init} (\mathcal{Q}) \cup \text{Ter} (\mathcal{Q}))$. A $\boldsymbol{(v, Z,l,m)}$\textbf{-fan} in a digraph $D$ consists of a collection of $l$ internally disjoint paths starting at a common vertex $v$, where each path terminates at a distinct vertex of $Z$ and every such path has length at most $m$. A $\boldsymbol{(Z,v,l,m)}$\textbf{-fan} is defined as above, except that the paths now start in $Z$ and terminate in $v$. 



 For a vertex $v \in V(D)$ and a subset $U$ of $V(D)$, we say that $\boldsymbol{(v, U)}$ \textbf{is $k$-connected} in $D$ if for any subset $S \subseteq V(D) \setminus \{v\}$ with $|S| \leq k - 1$, there exists a path from $v$ to a vertex in $U \setminus S$ in $D \setminus S$. Similarly, we say \textbf{$\boldsymbol{(U, v)}$ is $k$-connected} in $D$ if for any subset $S \subseteq V(D) \setminus \{v\}$ with $|S| \leq k - 1$, there exists a path from a vertex in $U \setminus S$ to $v$ in $D \setminus S$.

\section{Sparse spanning  $k$-strong oriented subdigraph}

This section aims to prove Theorem \ref{main1}. To this end, we first present several preliminaries.

\subsection{Preliminaries}

Menger's theorem is one of the cornerstone results in connectivity, and we use the following easy consequence of the theorem.
\begin{theorem}\label{menger}
\emph{\cite{menger}} Let $D$ be a $k$-strong digraph. Then for every choice of disjoint sets $\{x_1, \ldots$, $x_k\}$ and $\{y_1, \ldots, y_k\}$ of $V(D)$, there exist $k$ disjoint $(x_i, y_{\pi(i)})$-paths for some permutation $\pi$ of $\{1, 2, \ldots, k\}$.
\end{theorem}

We define the main tools, namely, ``nearly dominating vertex" and ``nearly dominating set". Let \( c \in \mathbb{N} \), and let \( u, v \) be two vertices of a digraph \( D \). We say $v$ is \textbf{$\boldsymbol{c}$-in-good for $\boldsymbol{u}$} in $D$ if either $v\rightarrow u$ or at least \( c \) internally disjoint \( (v, u) \)-paths of length 2 exist in \( D \). And we say $v$ is \textbf{$\boldsymbol{c}$-out-good for $\boldsymbol{u}$} in $D$ if either $u\rightarrow v$ or there exist at least $c$ internally disjoint $(u, v)$-paths of length 2 in $D$. A similar definition (i.e., $c$-good for $u$) has appeared in \cite{zhouDM26,zhou2025proof}.

 \begin{definition}\label{def2}
Given a digraph $D$, a vertex set $U$ is a \textbf{nearly in-dominating  set} of $D$ if, for every vertex $u\in U$ and every $c \in \mathbb{N}$, all but at most $2c$ vertices in $D \setminus U$ are $c$-in-good for $u$ in $D$. Symmetrically,  a vertex set $U$ is a \textbf{nearly out-dominating set} of $D$ if, for every vertex $u\in U$ and every $c \in \mathbb{N}$, all but at most $2c$ vertices in $D \setminus U$ are $c$-out-good for $u$ in $D$.
\end{definition}

We need the following lemma {from} \cite{zhou2025proof}, which establishes the existence of a  nearly in/out-dominating set in  semicomplete digraphs.

\begin{lemma}\label{key1} \emph{\cite{zhou2025proof}}
  Every semicomplete digraph $D$ {has}  a nearly out-dominating vertex {and} a nearly in-dominating vertex.
\end{lemma}

 {Lemma~\ref{nearly} yields the short disjoint paths required for the subsequent proof. To ensure that combining these paths with the previously constructed oriented subdigraph $R$ yields an oriented subdigraph, we require these paths to avoid all reversed arcs of $R$.}

\begin{lemma}\label{nearly}
Let $l_1$ be a positive integer, $T$ a tournament, $X_{\mathrm{in}}$ a nearly in‑dominating set, and $Y_{\mathrm{out}}$ a nearly out‑dominating set of $T$. Let $X\subseteq X_{\mathrm{in}}$, $Y\subseteq Y_{\mathrm{out}}$, $W\subseteq V(T)$, and let $R$ be a digraph \emph{(}not necessarily a subdigraph of $T$\emph{)} satisfying $\Delta^0(R)\leq l_1$. Then:
\begin{itemize}
    \item[(i)] for every $Z\subseteq V(T)\setminus X_{in}$ with
    $|Z|\geq 3|X|+5l_1+2|W|$,
    there exist $|X|$ disjoint paths of length at most $2$
    from $Z$ to $X$ in $T\setminus\rev{A(R)}$, whose internal vertices lie
    outside $W$;

    \item[(ii)] for every $Z\subseteq V(T)\setminus Y_{out}$ with
    $|Z|\geq 3|Y|+5l_1+2|W|$,
    there exist $|Y|$ disjoint paths of length at most $2$
    from $Y$ to $Z$ in $T\setminus\rev{A(R)}$, whose internal vertices lie
    outside $W$.
\end{itemize}
\end{lemma}

\begin{proof}
We prove (i); (ii) follows by reversing all arcs. Set $m:=|X|$ and
$H:=T\setminus\rev{A(R)}$.

Fix $x\in X$. Applying Definition~\ref{def2} with
$c=m+2l_1+|W|$
gives a set $Z'\subseteq Z$ such that
$|Z'|\geq |Z|-2(m+2l_1+|W|)\geq m+l_1$,
and, for every $v\in Z'$, either $vx\in A(T)$ or there are at least
$m+2l_1+|W|$ internally disjoint $2$-paths from $v$ to $x$. At most $l_1$ vertices $v\in Z'$ satisfy
$vx\in\rev{A(R)}$, since
$d^-_{\rev R}(x)=d_R^+(x)\leq l_1$. Hence there is a set
$Z_x\subseteq Z'$ with
$|Z_x|\geq m$
such that, for every $v\in Z_x$, either $vx\in A(H)$, or $vx\notin A(T)$.
In the latter case, among the
$m+2l_1+|W|$ internally disjoint $2$-paths from $v$ to $x$, at most
$2l_1$ use an arc of $\rev{A(R)}$ and at most $|W|$ have their internal
vertex in $W$. Thus $H$ contains at least
$m+2l_1+|W|-2l_1-|W|=m$
internally disjoint $2$-paths from $v$ to $x$ whose internal vertices lie
outside $W$. Therefore,
\begin{equation}\label{eq1}
    \begin{aligned}
        &\text{\ \ \ \ \ \ \ for every $x\in X$, there are at least $m$ vertices
$v\in Z$ }\\
&\text{\ \ \ \ such that either $vx\in A(H)$, or $H$ contains $m$ internally}\\
&\text{disjoint $2$-paths from $v$ to $x$ with
internal vertices outside $W$.}
    \end{aligned}
\end{equation}


We now construct the desired paths by induction. Suppose that
$\ell<m$ pairwise disjoint paths
$Q_1,\ldots,Q_\ell$
have been found in $H$, each of length at most $2$, from $Z$ to  $x_1,\ldots,x_\ell\in X$, with all internal vertices outside $W$.
By replacing a path with a suitable subpath if necessary, we may also assume that no internal vertex lies in $X\cup Z$. Choose
$x_{\ell+1}\in X\setminus\{x_1,\ldots,x_\ell\}$.
By \eqref{eq1}, since the paths $Q_1,\ldots,Q_\ell$ use only $\ell<m$ vertices
of $Z$, we may choose
$v\in Z\setminus\bigcup_{i=1}^{\ell}V(Q_i)$
satisfying \eqref{eq1} for $x_{\ell+1}$. If $vx_{\ell+1}\in A(H)$, we simply add the path $vx_{\ell+1}$. Otherwise, there are $m$
internally disjoint $2$-paths
$vu_jx_{\ell+1}$, $j\in[m]$,
in $H$, where the $u_j$ are distinct and lie outside $W$. If some $u_j\notin\bigcup_{i=1}^{\ell}V(Q_i)$, then we add $vu_jx_{\ell+1}$;
When $u_j\in Z$ or $u_j\in X$, we instead add the subpath $u_jx_{\ell+1}$ or $vu_j$, respectively. Hence we may assume that every $u_j$ lies on one of
$Q_1,\ldots,Q_\ell$.

Since $m>\ell$, some $Q_i$ contains at least two of the vertices $u_j$.
Write either
$Q_i=z_ix_i$
or
$Q_i=z_iw_ix_i$.
In the first case, necessarily $z_i,x_i\in\{u_j:j\in[m]\}$, and we replace
$Q_i$ by $vx_i$ and add a new path $z_ix_{\ell+1}$.
In the second case, at least two of $z_i,w_i,x_i$ belong to
$\{u_j:j\in[m]\}$. According to the pair chosen,
\begin{itemize}
    \item let $Q_i:=vw_ix_i$ and let $Q_{\ell+1}:= z_ix_{\ell+1}$, if  $z_i, w_i\in\{u_j:j\in[m]\}$;
    \item let $Q_i:=vx_i$ and let $Q_{\ell+1}:= z_iw_ix_{\ell+1}$, if  $w_i, x_i\in\{u_j:j\in[m]\}$;
    \item let $Q_i:=vx_i$ and let $Q_{\ell+1}:= z_ix_{\ell+1}$, if  $x_i, z_i\in\{u_j:j\in[m]\}$.
\end{itemize}
In each case, every arc of the two replacement paths belongs either to $Q_i$
or to one of the paths $vu_jx_{\ell +1}$. Hence both paths lie in $H$. They are vertex-disjoint from each other and from every $Q_j$ with $j\neq i$, have length at most $2$, and have no internal vertex in $W$. Thus any family of $\ell<m$ such paths can be augmented to one of size $\ell+1$. Repeating this yields $m=|X|$ pairwise vertex-disjoint paths from $Z$ to $X$, as required.
\end{proof}

 Lemma \ref{order-exists}, which is a key {tool} for the proof of Lemma \ref{good graph}, gives a nice ordering of the vertices of an oriented graph.
\begin{lemma}\label{order-exists}
  \emph{\cite{kang2017}} Let $D$ be an oriented graph of order $n$ with $\delta(D) \ge n - 1-s$. Then there exists an ordering $\sigma = (v_1, \dots, v_n)$ of $V(D)$ that satisfies the following:
\begin{enumerate}
\item[$(Q1_s)$] For any $i, j \in [n]$ with $i < j$, $v_i$ has at least $\frac{j-i-s}{2}$ out-neighbours in $\{v_{i+1}, \dots, v_j\}$;
\item[$(Q2_s)$] For any $i, j \in [n]$ with $i < j$, $v_j$ has at least $\frac{j-i-s}{2}$ in-neighbours in $\{v_i, \dots, v_{j-1}\}$.
\end{enumerate}
\noindent{}Moreover, an ordering $\sigma$ as above can be constructed in polynomial time.
\end{lemma}

Lemma \ref{lemma4} gives a {lower bound on the size of a} matching in a bipartite graph.

\begin{lemma}[See Claim 3.2 in \cite{kang2017}] \label{lemma4}

   {Let $s \ge 0$ {be} an integer and} let $G$ be a bipartite graph with bipartition $A \cup B$, where $A = \{a_1, \dots, a_n\}$ and $B = \{b_1, \dots, b_n\}$, satisfying the following conditions\emph{:}
\begin{enumerate}[label=(P\arabic*$_s$)]
    \item For all $i, j \in [n]$ with $i < j$,
    $|N_G(a_i) \cap \{b_{i+1}, \dots, b_j\}| \ge \frac{j-i-s}{2}$\emph{;}

    \item For all $i, j \in [n]$ with $i < j$,
    $|N_G(b_j) \cap \{a_i, \dots, a_{j-1}\}| \ge \frac{j-i-s}{2}$.
\end{enumerate}
Then $G$ contains a matching of size at least $n-s-1$.
\end{lemma}

Let $D$ be a digraph and let  $\sigma=(v_1,\dots,v_n)$ be an ordering of $V(D)$. An  arc $v_i\to v_j$ is called \textbf{$\boldsymbol{\sigma}$-forward} if $i<j$, and \textbf{$\boldsymbol{\sigma}$-backward} if $j<i$. Furthermore, for integers $1\leq a\leq b\leq n$ we let $\boldsymbol{\sigma(a,b)}:=\{v_\ell : a\le \ell\le b,\ \ell\in [n]\}$.  Thus $\sigma(a,b)=\emptyset$ if $a>b$.

\begin{definition}\label{def3}
For two nonnegative integers $c$ and $s$,  an $n$-vertex digraph $D'$ and an ordering $\sigma =(v_1,\dots,v_n)$ of $V(D')$, we say $D'$ is \textbf{$\boldsymbol{(\sigma,c,s)}$-nice} if it satisfies the following.
\begin{itemize}
    \item[$(D1)$] Every arc in $D'$ is $\sigma$-forward {In particular, $D'$ is acyclic.};
    \item[$(D2)$] For each vertex $v\in \sigma(n-2c - s+2, n) $, there are at most $2c+s$  arcs  entering $v$ in $D'$; For each vertex $v\in \sigma(1, 2c + s - 1)$, there are at most $2c+s$  arcs leaving $v$ in $D'$;
    \item[$(D3)$] Every vertex in $\sigma(1,n-2c-s+1)$ has out-degree at least $c$ in $D'$, and every vertex in $\sigma(2c+s,n)$ has in-degree at least $c$ in $D'$;
    \item[$(D4)$] $|A(D')| \le cn + 9c^2 + 5sc-4c$;
    \item[$(D5)$] For {every} $w \in V(D')$, both $(w, \sigma(n-2c - s+2, n) )$ and $(\sigma(1, 2c + s - 1), w)$ are $c$-connected in $D'$.
\end{itemize}
\end{definition}

{We shall use Lemma \ref{good graph} below
to control the number of arcs and meet degree constraints.} 
The main proof ideas follow from Claim 3.1 and Lemma 3.4 in \cite{kang2017}, and relevant results are also presented in \cite{kang2021}. To guarantee the property {that the digraph $D'$ below has no 2-cycles}, we make appropriate modifications to the original argument and derive a slightly stronger conclusion.

\begin{lemma}\label{good graph}  Let $c,s$ be nonnegative integers and let $n$ be a positive integer. Let $D$ be an oriented graph {on $n$ vertices} such that $\delta(D) \geq n - s - 1$. Then, {in polynomial time, we can construct }a spanning {oriented subdigraph} $D'$ of $D$ and an ordering $\sigma$ of $V(D)$ such that $D'$ is $(\sigma,c,s)$-nice.
\end{lemma}

\begin{proof}
{ If $c=0$, take $\sigma$ to be an arbitrary ordering and let $D'$ be the arcless spanning digraph. Conditions (D1)--(D4) are immediate and (D5) is vacuous; hence the conclusion holds. We may therefore assume that $c\geq1$.\unboldmath}   If $n \leq 2c+s-1$, then take $\sigma$ to be an arbitrary ordering of $V(D)$ and let $D'$ be the arcless digraph on $V(D)$. The conclusion holds trivially. Thus we may assume that $n \ge 2c+s $. By Lemma \ref{order-exists}, we can find an ordering $\sigma = (v_1, \dots, v_n)$ of $V(D)$ which satisfies conditions ($Q1_s$) and ($Q2_s$).
   We {now} consider an auxiliary bipartite graph $H_0$ with a bipartition $A \cup B$, where $A = \{v_1, \dots, v_n\}$ and $B = \{v_1', \dots, v_n'\}$, such that $v_i v_j' \in H_0$ if and only if $v_i v_j$  is a $\sigma$-forward arc of $D$ (i.e., $i < j$ and $v_i v_j \in A(D)$). Note that the conditions (Q1$_s$)-(Q2$_s$) imply that the graph $H_0$ satisfies conditions (P1$_s$)-(P2$_s$).

Assume that we have already constructed a subgraph $H_\ell$ satisfying conditions (P1$_{s+2\ell}$)-(P2$_{s+2\ell}$). By Lemma \ref{lemma4} with $G:=H_\ell$, $H_\ell$ contains a matching $M_\ell$ of size at least $n -s- 2\ell - 1$. Let $H_{\ell+1} := H_\ell \setminus M_\ell$. Then for any $i, j \in [n]$, we have $|N_{H_\ell}(v_i) \setminus N_{H_{\ell+1}}(v_i)| \le 1$ and $|N_{H_\ell}(v'_j) \setminus N_{H_{\ell+1}}(v'_j)| \le 1$. Thus the graph $H_{\ell+1}$ satisfies conditions (P1$_{s+2\ell+2}$)-(P2$_{s+2\ell+2}$). Repeating this process for $0 \le \ell \le c-1$ provides edge-disjoint matchings $M_0, M_1, \dots, M_{c-1}$ of $H_0$ where the size of $M_\ell$ is at least $n  -s- 2\ell - 1$ for $0 \le \ell \le c-1$. By deleting some edges, we may assume that
\begin{equation}\label{eee1}
    |E(M_\ell)| = n - s - 2\ell - 1 \text{ for }0 \le \ell \le c-1.
\end{equation}
Let $M$ be the subgraph of $H_0$ such that $E(M) := \bigcup_{\ell=0}^{c-1} E(M_\ell)$ and let $D_1$ be the spanning subdigraph of $D$ such that
\begin{equation}\label{eqq1}
    \begin{aligned}
    A(D_1):=&\{v_iv_j\in A(D)\mid v_iv_j'\in E(M),\  i\notin [1,2c+s-1] \text{ {and} } j\notin [n-2c-s+2,n]\}.
\end{aligned}
\end{equation}
Then by construction of $H_0$, every arc of $D_1$ is a $\sigma$-forward arc and
\begin{equation}\label{eq:3.2}
    \Delta(M) \leq {c} \quad \text{and} \quad |E(M)| = \sum_{\ell=0}^{{c-1}} |E(M_\ell)| \stackrel{(\ref{eee1})}{=} cn - c^2 - sc.
\end{equation}
Also this implies that
\begin{equation}\label{eq:3.3}
    \begin{aligned}
        \Delta^+(D_1) \leq c, \  \Delta^-(D_1) \leq c,\
        d^-_{D_1}(v_i) \leq \min\{c, i-1\}  \text{ and }  d^+_{D_1}(v_i) \leq \min\{c, n-i\},
    \end{aligned}
\end{equation}
and thus
\begin{equation}\label{eq:3.4}
cn - c^2 - sc\geq   |A(D_1)| \stackrel{(\ref{eqq1})}{\geq}  cn - c^2 - sc - 2(2c+s-1)c=cn-5c^2-3sc+2c.
\end{equation}

 By (Q2$_s$), for each {vertex $v_i$ with} $2c + s \leq i \leq n$, the number of $\sigma$-forward arcs  from $ \{v_{i-2c-s+1},\ldots$, $v_{i-1}\} $ to $v_i$ is at least $\left\lceil \frac{i-(i-2c-s+1)-s}{2} \right\rceil  {=} {c}$. Similarly, by (Q1$_s$),for each $1\leq i \leq n-2c-s+1 $, the number of $\sigma$-forward arcs from $v_i$ to $ \{v_{i+1},\ldots, v_{i+2c+s-1}\} $ is at least $\left\lceil \frac{i+2c+s-1-i-s}{2} \right\rceil  {=} {c}$. Next, we select the sets ${A_i^-}$ and ${A_i^+}$ in two successive steps. 
 Firstly, for each $2c + s \leq i \leq n-2c-s+1$, we can choose a set ${A_i^-}$ of $\sigma$-forward arcs from $ \sigma (i-2c-s+1, i-1) $ to $v_i$ such that ${A_i^-} \subseteq A(D) \setminus A(D_1)$ and $|{A_i^-}| = c - d^-_{D_1}(v_i)$. Then choose a set ${A_i^+}$ of $\sigma$-forward arcs from $v_i$ to $ \sigma (i+1,  i+2c+s-1) $ such that ${A_i^+}  \subseteq A(D) \setminus A(D_1)$ and $|{A_i^+}| = c - d^+_{D_1}(v_i)$. Define the spanning subdigraph $D_2 $ of $ D$ {by setting}
\[
\begin{aligned}
    A(D_2) := A(D_1) \cup \bigcup_{i=2c+s}^{n-2c-s+1} ({A_i^-} \cup {A_i^+}).
\end{aligned}
\]
Secondly, for each $1 \leq i \leq \min \{2c + s-1, n-2c-s+1\}$, we choose a set ${A_i^+}$ of $\sigma$-forward arcs from $v_i$ to $ \sigma (i+1,i+2c+s-1) $ such that ${A_i^+} \subseteq A(D) \setminus A(D_2)$ and $|{A_i^+}| =\max \{0, c - d^+_{D_2}(v_i)\}$. And for each $\max \{2c + s, n-2c-s+2\} \leq i \leq n$, we choose a set ${A_i^-}$ of $\sigma$-forward arcs  {entering} $v_i$  {from} $ \sigma (i-2c-s+1, i-1) $ such that ${A_i^-} \cap A(D_2) = \emptyset$ and $|{A_i^-}| =\max \{0, c - d^-_{D_2}(v_i)\}$.  All the remaining undetermined \({A_i^+}\) and \({A_i^-}\) are defined as empty sets. Define {the} spanning digraph $D' $ of $ D$ {by setting}
\[
\begin{aligned}
    A(D') :=A(D_1) \cup \bigcup_{i=1}^{n} ( {A_i^+}\cup {A^-_i}).
\end{aligned}
\]

Note that, since every arc of $D'$ is a forward arc with respect to $\sigma$, $D'$ is an {acyclic} oriented graph. One can easily check that (D1) and (D3) in Definition \ref{def3} hold. Moreover, by the definition of $D_1$, $d^+_{D_1}(v_i)=0$ for $1\le i\le 2c+s-1$, and
$d^-_{D_1}(v_i)=0$ for $n-2c-s+2\le i\le n$. 
 Furthermore, every arc contained in some ${A_i^+}$ or ${A_i^-}$ joins two vertices whose indices differ by at most $2c+s-1$. Hence
$d^+_{D'}(v_i)\le 2c+s-1$ for $1\le i\le 2c+s-1$,
and
$d^-_{D'}(v_i)\le 2c+s-1$ for $n-2c-s+2\le i\le n$. Thus $(D2)$ holds. Since the ordering in Lemma \ref{order-exists} can be found in polynomial time and each matching $M_\ell$ can be obtained by a standard polynomial-time algorithm for maximum matching in bipartite graphs, the whole construction can be carried out in polynomial time.

We start by bounding the number of arcs in $D'$. Since {$A(D_1)\subseteq A(D_2)$} and $\Delta^-(D_1)\le c$, for every
$i\in[n]$ we have
\[
\max\{0,c-d^-_{D_2}(v_i)\}
\le c-d^-_{D_1}(v_i).
\]
Consequently,
\begin{align*}
\left|\bigcup_{i=1}^{n}{A_i^-}\right|\le
\sum_{i=1}^{n}\left(c-d^-_{D_1}(v_i)\right)=
cn-\sum_{i=1}^{n}d^-_{D_1}(v_i)=
cn-|A(D_1)| \stackrel{(\ref{eq:3.4})}{\leq }
5c^2+3sc-2c,
\end{align*}
where we used $|A(D_1)| = \sum_{i=1}^{n} d_{D_1}^{-}(v_i)$. Similarly, we also have $\left|\bigcup_{i=1}^{n} {A_i^+}\right| \leq 5c^2+3sc-2c$. Thus we have
\[
\begin{aligned}
|A(D')| &\leq |A(D_1)| + \left|\bigcup_{i=1}^{n} ({A_i^-}\cup {A_i^+})\right| \\
&\stackrel{(\ref{eq:3.4})}{\leq} cn - c^2 - sc + 2(5c^2+3sc-2c ) = cn + 9c^2 + 5sc-4c.
\end{aligned}
\]
Hence {(D4)} holds.

We now verify {that the} digraph \(D'\) {satisfies} ({D5)}, which is to prove that for any subset $S \subseteq V(D) $ with $|S| \leq c - 1$ and any vertex $w\in V(D')\setminus S$, there exists a path $P$ from $w$ to a vertex in $\sigma(n-2c - s+2, n)  \setminus S$ in $D' \setminus S$. Similarly, there exists a path $P'$ from a vertex in $\sigma(1, 2c + s - 1)\setminus S$ to $w$ in $D' \setminus S$.  Take a path $P$ starting at $w$ and ending at $v_j$ with the largest possible $j$ in the ordering  $\sigma$. Suppose   $j \leq  n - 2c - s +1 $, i.e., $ v_j\notin \sigma(n-2c - s+2, n) $. By (D3), $d^+_{D'}(v_j)\ge c$.
Since every arc of $D'$ is $\sigma$-forward, every out-neighbour
of $v_j$ has index larger than $j$. As $|S|\le c-1$, there exists
$v_{j'}\in N^+_{D'}(v_j)\setminus S$ with $j'>j$.
This contradicts the maximality of $j$. Thus we have $v_j\in \sigma(n-2c - s+2, n) $. Therefore there exists a path $P$ in $D'-S$ from $w$ to $v_j \in \sigma(n-2c - s+2, n) $. We can find $P'$ in a similar way, which completes the proof.
\end{proof}

\subsection{Proof of Theorem \ref{main1}}

In this subsection, suppose $D=(V_1,V_2; A)$ is a $k$-strong split digraph satisfying { $\delta^0(D)\geq 26k+15$}. We first give the outline of the proof.

\textbf{Sketch of Proof.}
 We first select two disjoint nearly dominating sets $X$ and $Y$ from a spanning tournament $T$ of $D[V_2]$, and then establish $k$ disjoint paths $\mathcal{P}$ from some {set of $k$} vertices in $X$ to some {set of $k$} vertices in $Y$ via Menger’s theorem. By reversing certain arcs of $A(T)$, we obtain a new tournament $T_{\mathrm{rev}}$ containing all {arcs of $\mathcal P$ whose two ends lie in $V_2$}. We use Lemma \ref{good graph} to obtain three sparse oriented subdigraphs $T'_{\text{sparse}}, T''_{\text{sparse}}$ and $T^{\text{abs}}_{\text{sparse}}$ of $T_{\mathrm{rev}}$. The constructed subdigraph $T^{\text{abs}}_{\text{sparse}}$ possesses two small absorbing sets $S_{\mathrm{in}}$ and $S_{\mathrm{out}}$ to control the connectivity properties of vertices in $V_2\setminus (X\cup Y)$. To obtain a highly strong spanning subdigraph, we iteratively construct two families of auxiliary fans and maintain the entire structure $2$-cycle-free, which is made possible by the large minimum semidegree of $D$. Finally, we greedily add arcs incident to vertices in $V_1\cup X\cup Y$, using the large minimum semidegree condition.  Integrating all constructed components yields the desired spanning $k$-strong oriented subdigraph, and the upper bound on its arc cardinality is consequently established.

\begin{figure}[htbp]
    \centering
\includegraphics[width=0.6\textwidth]{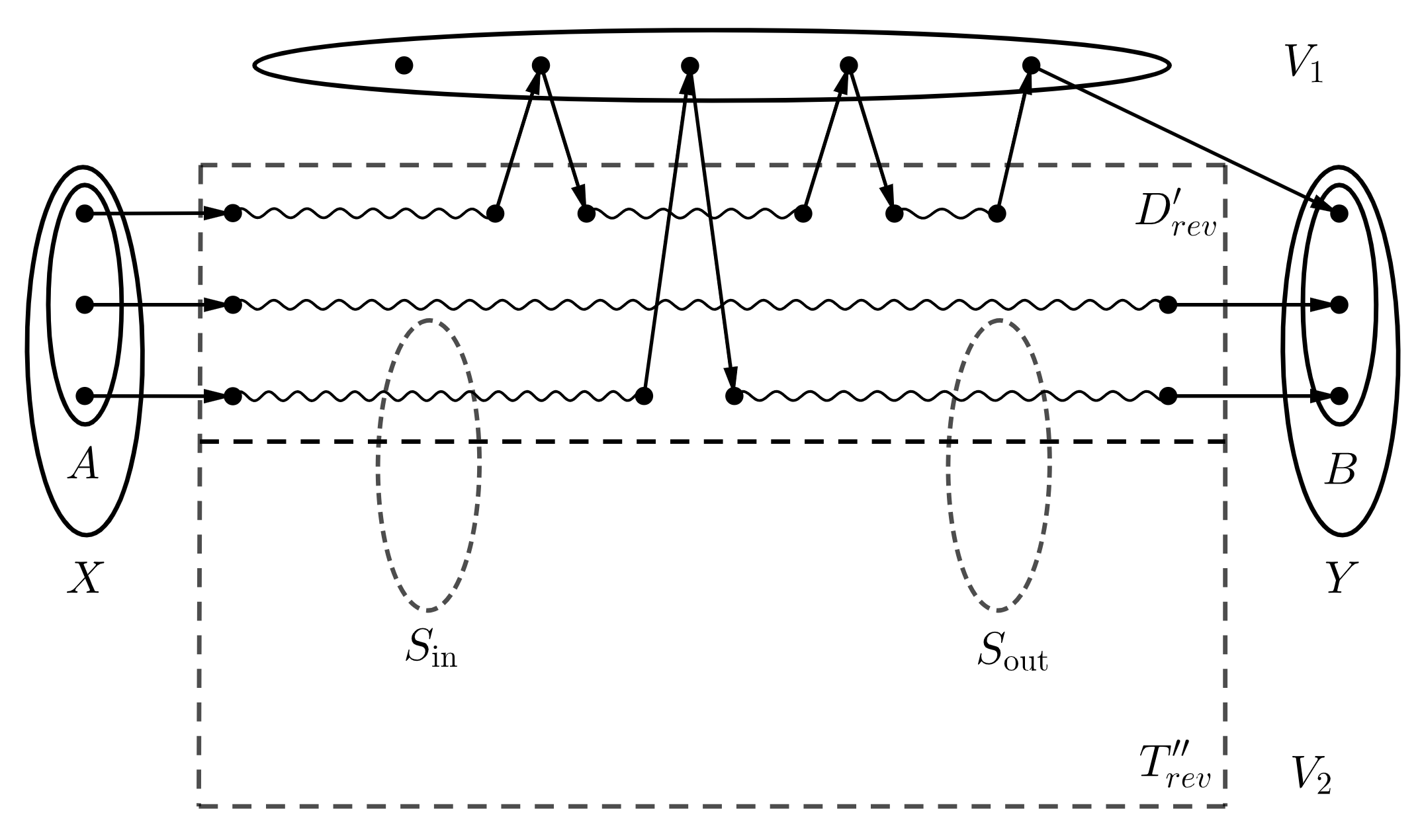}
    \caption{Structural framework of sparse spanning $k$-strong oriented subdigraphs of a split digraph $D$.}
    \label{fig1}
\end{figure}


\begin{proof}
Let \(T\) be an arbitrary spanning tournament of \(D[V_2]\).
 {As all neighbours of vertices in $V_1$ lie in $V_2$, we have }{ $|T|\geq 26k+15$}. Since every induced subdigraph of a tournament is still semicomplete,  {repeated applications of Lemma  \ref{key1} gives a nearly in-dominating set $X\subseteq V(T)$ of size $3k+3$, and  then a nearly out-dominating set $Y\subseteq V(T)\setminus X$ of size $3k+3$ in $T-X$.}  These sets can be found in polynomial time.
The cardinality $3k+3$ is selected to guarantee that \(T[X]\) contains a set \(A = \{x_1, \dots, x_k\} \subseteq X\) with \(d^-_{T[X]}(x_i) \geq k+1\) for each \(x_i \in A\).  Similarly, \(T[Y]\) contains a set \(B = \{y_1, \dots, y_k\} \subseteq Y\) with \(d^+_{T[Y]}(y_i) \geq k+1\) for each \(y_i \in B\).  To {show} this, suppose that fewer than $k$ vertices of $T[X]$ have in-degree at least $k+1$. Then some set $U\subseteq X$ of size $2k+4$ consists entirely of vertices of in-degree at most $k$ in $T[X]$. Hence $\binom{2k+4}{2}=|A(T[U])|\leq\sum_{u\in U}d^-_{T[X]}(u)\leq k(2k+4)$, a contradiction. The assertion for $T[Y]$ follows symmetrically.

 Since \(D\) is \(k\)-strong,  Menger's theorem {implies that,} after relabeling \(A\) and \(B\) if necessary, there exist \(k\) disjoint paths \(\mathcal{P} = \{P_1, \dots, P_k\}\) from \(A\) to \(B\) such that \(P_i\) is a minimal  \((x_i, y_i)\)-path in \(D\). Clearly,
\begin{equation}\label{e5}
    \text{for every } v \in V(D),  {\text{ the set }} \rev{A(\mathcal{P})} \text{ contains at most one out-arc and one in-arc of } v.
\end{equation}
 { This implies that $\delta^{0}(D\setminus \overleftarrow{A(\mathcal{P})})\geq (26k+15)-1=26k+14$.\unboldmath}

 {Let $A_2(\mathcal P):=A(\mathcal P)\cap A(D[V_2])$, and define}
\[
 {A(T_{\mathrm{rev}}):=\bigl(A(T)\setminus \rev{A(\mathcal{P})} \bigr)\cup A_2(\mathcal P).}
\]
 {Then $T_{\mathrm{rev}}$ is a spanning tournament of $D[V_2]$.}
By the definition of $\rev{A(\vP)}$, $T_{\text{rev}}$ is defined to preserve the property of being 2-cycle-free {and it contains all those arcs of $A(\vP)$ which have both endvertices in $V_2$}. Indeed, all subsequent fan constructions utilize arcs from the original tournament $T$ whenever possible. {Next we}  partition $V_2$ into three parts $V^1_2:=X\cup Y$, $V^2_2:=\text{Int}(\vP)\cap (V_2\setminus (X\cup Y))$ and $V^3_2:=V_2\setminus ({V_2^1}\cup V^2_2) $. Let $D'_{\mathrm{rev}}:=T_{\mathrm{rev}}[ V^2_2]\setminus A(\vP)$ and let $T''_{\mathrm{rev}}:= T_{\mathrm{rev}}[V^3_2]$.  One can easily check that $D'_{\mathrm{rev}}$ is an oriented graph with $\delta (D'_{\mathrm{rev}}) \geq {|V_2^2|}-3$ and $T''_{\mathrm{rev}}$ is a tournament.

For each {of the} oriented graphs $D'_{\mathrm{rev}}$ and  $T''_{\mathrm{rev}}$, we {construct} a spanning sparse subdigraph by using Lemma~\ref{good graph} with different parameters, respectively.  If either graph is empty, we take its sparse spanning subdigraph and its ordering to be empty.   Specifically, applying Lemma~\ref{good graph} to $D'_{\mathrm{rev}}$ with $c:=k-1$  and $s:=2$, {we obtain} a spanning oriented graph $T'_{\mathrm{sparse}}$ of $D'_{\mathrm{rev}}$ {and an ordering $\sigma '_{\mathrm{sparse}}$ of $V(D'_{\mathrm{rev}})$ such that $T'_{\mathrm{sparse}}$ is $(\sigma '_{\mathrm{sparse}},k-1,2)$-nice.}  The case $k=1$, where $c=0$, is also covered by Lemma~\ref{good graph}.
 Hence, (D4) gives
\begin{equation}\label{degree1}
 |A(T'_{\mathrm{sparse}})| \le (k-1)|T'_{\mathrm{sparse}}| +9k^2-12k+3.
 \end{equation}
{Let $S'_{\mathrm{out}}:=\sigma '_{\mathrm{sparse}}(1,2c + s-1) $ and let $S'_{\mathrm{in}}:=\sigma '_{\mathrm{sparse}}(|D'_{\mathrm{rev}}|-2c - s+2, |D'_{\mathrm{rev}}|) $. Then  $|S'_{\mathrm{in}}|,| S'_{\mathrm{out}}|\leq 2c+s-1=  {2k-1}$ and (D5) implies that for every $u \in V(D'_{\mathrm{rev}})$,}
\begin{equation}\label{e7}
\begin{aligned}
     \text{ both $(u, S'_{\mathrm{in}})$ and $(S'_{\mathrm{out}}, u)$}\text{  are $(k-1)$-connected in $T'_{\mathrm{sparse}}$}.
\end{aligned}
\end{equation}
 Applying Lemma~\ref{good graph} to $T''_{\mathrm{rev}}$ with $c:=k$ and $s:=0$, we obtain a spanning oriented graph $T''_{\mathrm{sparse}}$ of $ T''_{\mathrm{rev}}$ {and an ordering $\sigma ''_{\mathrm{sparse}}$ of $V(T''_{\mathrm{rev}})$ such that $T''_{\mathrm{sparse}}$ is $(\sigma ''_{\mathrm{sparse}},k,0)$-nice. Hence, (D4) gives}
 \begin{equation}\label{degree2}
|A(T''_{\mathrm{sparse}})| \le k|T''_{\mathrm{sparse}}| + 9k^2-4k.
 \end{equation}
{Let $S''_{\mathrm{out}}:=\sigma ''_{\mathrm{sparse}}(1,2c + s-1) $ and let $S''_{\mathrm{in}}:=\sigma ''_{\mathrm{sparse}}(|T''_{\mathrm{rev}}|-2c - s+2, |T''_{\mathrm{rev}}|) $. Then  $|S''_{\mathrm{in}}|,| S''_{\mathrm{out}}|\leq 2c+s-1=  {2k-1}$ and (D5) implies that} 
\begin{equation}\label{e9}
\begin{aligned}
     &\text{for every } u \in V(T''_{\mathrm{rev}}),   \text{  both $(u, S''_{\mathrm{in}})$ and $(S''_{\mathrm{out}}, u)$ are $k$-connected in $T''_{\mathrm{sparse}}$}.
\end{aligned}
\end{equation}
Consider the tournament $T^{\mathrm{abs}}:= T_{\mathrm{rev}}[ S'_{\mathrm{in}}\cup S''_{\mathrm{in}}\cup S'_{\mathrm{out}}\cup S''_{\mathrm{out}} ] $.  Since
{ 
\[
|V_2\setminus(X\cup Y)|\geq 26k+15-2(3k+3)=20k+9,
\]
\unboldmath}
{  at least one of $V_2^2$ and $V_2^3$ satisfies $|V_2^i|\geq \lceil(20k+9)/2\rceil>4k-2$.\unboldmath} Accordingly, either $|S'_{\mathrm{in}}|=2k-1$ and $|S'_{\mathrm{out}}|=2k-1$, or $|S''_{\mathrm{in}}|=2k-1$ and $|S''_{\mathrm{out}}|=2k-1$. Consequently,
\begin{equation}\label{Tabssize}
    4k-2\leq |T^{\mathrm{abs}}|\leq 8k-4.
\end{equation}
 Applying Lemma~\ref{good graph} to $T^{\mathrm{abs}}$ with $c:=k$ and $s:=0$, we obtain a spanning oriented graph $T^{\mathrm{abs}}_{\mathrm{sparse}}$ of $ T^{\mathrm{abs}}$, {and an ordering $\sigma ^{\mathrm{abs}}_{\mathrm{sparse}}$ of $V(T^{\mathrm{abs}})$ such that $T^{\mathrm{abs}}_{\mathrm{sparse}}$ is $(\sigma ^{\mathrm{abs}}_{\mathrm{sparse}},k,0)$-nice.} Hence, (D4) gives
 \begin{equation}\label{degree3}
      |A(T^{\mathrm{abs}}_{\mathrm{sparse}})| \le k|T^{\mathrm{abs}}_{\mathrm{sparse}}| +9k^2 - 4k\leq 17k^2-8k.
 \end{equation}
Let $S_{\mathrm{out}}:=\sigma ^{\mathrm{abs}}_{\mathrm{sparse}}(1,2k-1) $ and let $S_{\mathrm{in}}:=\sigma ^{\mathrm{abs}}_{\mathrm{sparse}}(|T^{\mathrm{abs}}|-2k+2, |T^{\mathrm{abs}}|) $. 
Then {(\ref{Tabssize}) implies that $S_{\mathrm{in}}\cap S_{\mathrm{out}}=\emptyset$, $|S_{\mathrm{in}}|,| S_{\mathrm{out}}|=  {2k-1}$ and  }(D5) implies that
\begin{equation}\label{e11}
\begin{aligned}
     &\text{for every } u \in V(T^{\mathrm{abs}}_{\mathrm{sparse}}),   \text{  both $(u, S_{\mathrm{in}})$ and $(S_{\mathrm{out}}, u)$ are $k$-connected in $T^{\mathrm{abs}}_{\mathrm{sparse}}$}.
\end{aligned}
\end{equation}
{Let }$D^0$ be the spanning subdigraph of $D$ with {arc set}
\[
 A(D^0)=A(\mathcal{P}) \cup A(T'_{\text{sparse}})\cup A(T''_{\text{sparse}})\cup A(T^{\text{abs}}_{\text{sparse}}).
\]
Since $ T'_{\text{sparse}}$, $ T''_{\text{sparse}} $, and $T^{\text{abs}}_{\text{sparse}}$ are subdigraphs of the tournament $ T_{\mathrm{rev}} $, it follows that $ D^0$ is {an oriented graph}.
For each $v\in S_{\mathrm{in}}\cup S_{\mathrm{out}}$, by (D2), there are at most $2k$ in-arcs and at most $2k$ out-arcs of $v$ in $T^{\text{abs}}_{\text{sparse}}$.  Since $V(T'_{\text{sparse}})\cap V(T''_{\text{sparse}})=\emptyset$, $v$ belongs to exactly one of $T'_{\text{sparse}}$ or $T''_{\text{sparse}}$, so  (D2) again yields that there are at most $2k$ in-arcs and at most $2k$ out-arcs of  $v$. The set $A(\mathcal{P})$ contains at most one out‑arc and one in‑arc of $v$ by \eqref{e5}.    Hence,
\begin{equation}\label{D0}
\begin{aligned}
     &\text{$D^0$ contains at most $ 4k+1$ in-arcs and at most $4k+1$ out-arcs of $ v\in S_{\mathrm{in}}\cup S_{\mathrm{out}}$,}\\
        &\text{ \ \ \  ${D^0}$ contains at most one in-arc and out-arc of $ v$ for each $ v\in V_1\cup X\cup Y$}.
     \end{aligned}
 \end{equation}

 To guarantee the $k$-strong connectivity of the final spanning oriented graph, we construct suitable fans within $D\setminus \overleftarrow{A(D^0)}$ for every vertex $v\in S_{\mathrm{in}}\cup S_{\mathrm{out}}$. To be precise, for each vertex $v\in S_{\mathrm{in}}$, we build  {  an oriented} $(v,X,3k+3,4)$-fan $\mathcal{Q}_{v,X}$ and set $\mathcal{Q}_{Y,v}=\emptyset$; For each vertex $v\in S_{\mathrm{out}}$, we build  {  an oriented} $(Y,v,3k+3,4)$-fan $\mathcal{Q}_{Y,v}$ and set $\mathcal{Q}_{v,X}=\emptyset$. {  For each vertex $u_l$, exactly one of $\mathcal{Q}_{u_l,X}$ and $\mathcal{Q}_{Y,u_l}$ is nonempty as $S_{\mathrm{in}}\cap S_{\mathrm{out}}=\emptyset$. Our construction proceeds iteratively: at each step $l$, we build either the oriented fan $\vQ_{u_l,X}$ or $\vQ_{Y,u_l}$, such that all arcs of the newly constructed fan avoid every reversed arc in $\bigcup_{j\in[l-1]}\overleftarrow{A(\mathcal{Q}_{u_j,X}\cup\mathcal{Q}_{Y,u_j})}$. This iterative procedure then yields that the whole arc set $\bigcup_{j\in[m]}A(\mathcal{Q}_{u_j,X}\cup\mathcal{Q}_{Y,u_j})$ is free of $2$-cycles.}
 Set $m:=|S_{\mathrm{in}}\cup S_{\mathrm{out}}|\leq 4k-2$ and write $S_{\mathrm{in}}\cup S_{\mathrm{out}}=\{u_1,u_2,\dots,u_{m}\}$. We proceed to construct these fans sequentially in the given order. Assume that we have already established the fans $\mathcal{Q}_{u_1,X},\mathcal{Q}_{Y,u_1},\dots,\mathcal{Q}_{u_{l-1},X},\mathcal{Q}_{Y,u_{l-1}}$ for vertices $u_1,\dots$, $u_{l-1}$ $\in S_{\mathrm{in}}\cup S_{\mathrm{out}}$ so that these satisfy the following:

 \begin{itemize}
    \item[(I$_{l-1}$)] The vertex $u_l$ has at least {$18k+14$} out-neighbors and in-neighbors in $D_{l-1}$, where
    \[
    D_{l-1}:=D\setminus \bigg(\overleftarrow{A(D^0)}\cup \bigcup_{j\in [l-1]}\overleftarrow{A(\mathcal{Q}_{u_j,X}\cup \mathcal{Q}_{Y,u_j})}\bigg);
    \]
    \item[(II$_{l-1}$)] The arc set $\bigcup_{j\in [l-1]}\overleftarrow{A(\mathcal{Q}_{u_j,X}\cup \mathcal{Q}_{Y,u_j})}$ uses at most $|S_{\mathrm{in}}|+| S_{\mathrm{out}}|-1\leq 4k-3$ in-arcs and $|S_{\mathrm{in}}|+| S_{\mathrm{out}}|-1\leq 4k-3$ out-arcs for each vertex in $ V(D)\setminus \{u_1,\ldots,u_{l-1}\}$;
    \item[(III$_{l-1}$)] For all $i<l$, $\mathcal{Q}_{u_i,X}$ forms a $(u_i,X,3k+3,4)$-fan in $D_{i-1}$ whenever $u_i\in S_{\mathrm{in}}$, and $\mathcal{Q}_{Y,u_i}$ forms a $(Y,u_i,3k+3,4)$-fan in $D_{i-1}$ whenever $u_i\in S_{\mathrm{out}}$.
     \item[(IV$_{l-1}$)] For all $i<l$, the union $\mathcal{Q}_{u_i,X}\cup\mathcal{Q}_{Y,u_i}$ contains no $2$-cycles.
     \item[(V$_{l-1}$)] For all $i<l$, among arcs in $A(\mathcal{Q}_{u_i,X})$, only those  from $u_i$ to $V_2\setminus Y$ {may be} reverse arcs of arcs of $A(T)$; the rest arcs are either between $V_1$ and $V_2$ or arcs in $T\setminus A(\mathcal{P})$.

     \item[(VI$_{l-1}$)] For all $i<l$, among arcs in $A(\mathcal{Q}_{Y,u_i})$, only those from $V_2\setminus X$ to $u_i$ {may be} reverse arcs of $A(T)$; the rest are either cross arcs between $V_1$ and $V_2$ or arcs in $T\setminus A(\mathcal{P})$.
\end{itemize}
For $l=1$, (II$_0$)-(VI$_0$) are vacuous, while (I$_0$) follows from {$\delta^0(D)\ge26k+15$} and \eqref{D0}.


Without loss of generality, we may assume that \(u_l\in S_{\mathrm{in}}\). {  The case \(u_l\in S_{\mathrm{out}}\) follows by symmetric  arguments. For completeness, we nevertheless provide the construction of \(\vQ_{Y,u_l}\) if \(u_l\in S_{\mathrm{out}}\).  \unboldmath} We {now describe how to} construct a $(u_l,X,3k+3,4)$-fan $\mathcal{Q}_{u_l,X}=\{Q_{u_l,x_1},Q_{u_l,x_2},\dots$, $Q_{u_l,x_{3k+3}}\}$ in  the digraph $D_{l-1}$. For each vertex $x_i\in X$, the directed path $Q_{u_l,x_i}$ is specified as below. If the arc $u_lx_i\in A(D_{l-1})$, let $Q_{u_l,x_i}=u_lx_i$ be a trivial path of length 1. Otherwise, we take $Q_{u_l,x_i}$ to be a directed path whose length is bounded by 4. Denote by $X':=\{x_1,x_2,\dots,x_t\}$ the set of those vertices in $X$ satisfying $u_lx_i\in A(D_{l-1})$, and let $X''=X\setminus X' $. We start by  finding a $(u_l,V_2\setminus (X\cup Y\cup S_{\mathrm{out}}),{ {13k+12-t}},2)$-fan in $D_{l-1}$. By (I$_{l-1}$) and the definition of $X'$, the number of out-neighbors of $u_l$ in $D_{l-1}\setminus (X\cup Y\cup S_{\mathrm{out}})$ is at least
{\[
\begin{aligned}
   18k+14-|X'|-|Y\cup S_{\mathrm{out}}|&\geq 18k+14-t-(3k+3)-(2k-1)={  13k+12-t}.
\end{aligned}
\]}
Choose ${  13k+12-t}$ such out-neighbors and denote them by $U^+$. Use each vertex of $U^+\cap V_2$ directly as a terminal. For each vertex $w\in U^+\cap V_1$, by \eqref{D0} and (II$_{l-1}$), the number of out-neighbours of $w$ in $D_{l-1}\setminus (X\cup Y\cup S_{\mathrm{out}})$ is at least
{ {\begin{equation*}
    \begin{aligned}
        & \ \delta^0(D)-d^-_{D_0}(w)-(4k-3) -|X\cup Y\cup S_{\mathrm{out}}|\\
        &\geq 26k+15-1-(4k-3)-2(3k+3)-(2k-1)=14k+12>{  13k+12-t}.
    \end{aligned}
\end{equation*}}}
{Thus for}  $U^+\cap V_1$ {we can choose} distinct terminals greedily in $V_2\setminus(X\cup Y\cup S_{\mathrm{out}})$. {This completes the construction of the desired $(u_l,V_2\setminus (X\cup Y\cup S_{\mathrm{out}}),{  13k+12-t},2)$-fan  which we denote by $F^+_{u_l}$ in $D_{l-1}$}. Let $Z^+_l$ denote the ${  13k+12-t} $ terminal vertices of
 $F^+_{u_l}$. 
Recall that $X$ is nearly in-dominating and observe that
\[
\begin{aligned}
3|X''|+5+2|S_{\mathrm{in}}\cup X'|
&\leq 3(3k+3-t)+5+2((2k-1)+t)= {  13k+12-t}=|Z_l^+|.
\end{aligned}
\]
{Thus we can apply Lemma~\ref{nearly} with $T:=T$, $R:=\mathcal{P}$, $W:=S_{\mathrm{in}}\cup X'$ and $l_1:=1$, to conclude that}
there exist a set $\mathcal{Q}'_l$  of $3k+3-t $ disjoint paths from $Z^+_l$ to $X''$ of each length at most 2 in $T\setminus \rev{A(\mathcal{P})}$, which are internally disjoint from {$S_{\mathrm{in}}\cup X'$}. {Combining $F^+_{u_{l}}$, $\mathcal{Q}'_l$ and the direct arcs from $u_{l}$ to $X'$, we obtain a} $( u_l,X,3k+3,4)$-fan  $\vQ_{u_l,X}$ (see Fig. \ref{fan:fig1}).


\begin{figure}[htbp]
    \centering
\includegraphics[width=0.6\textwidth]{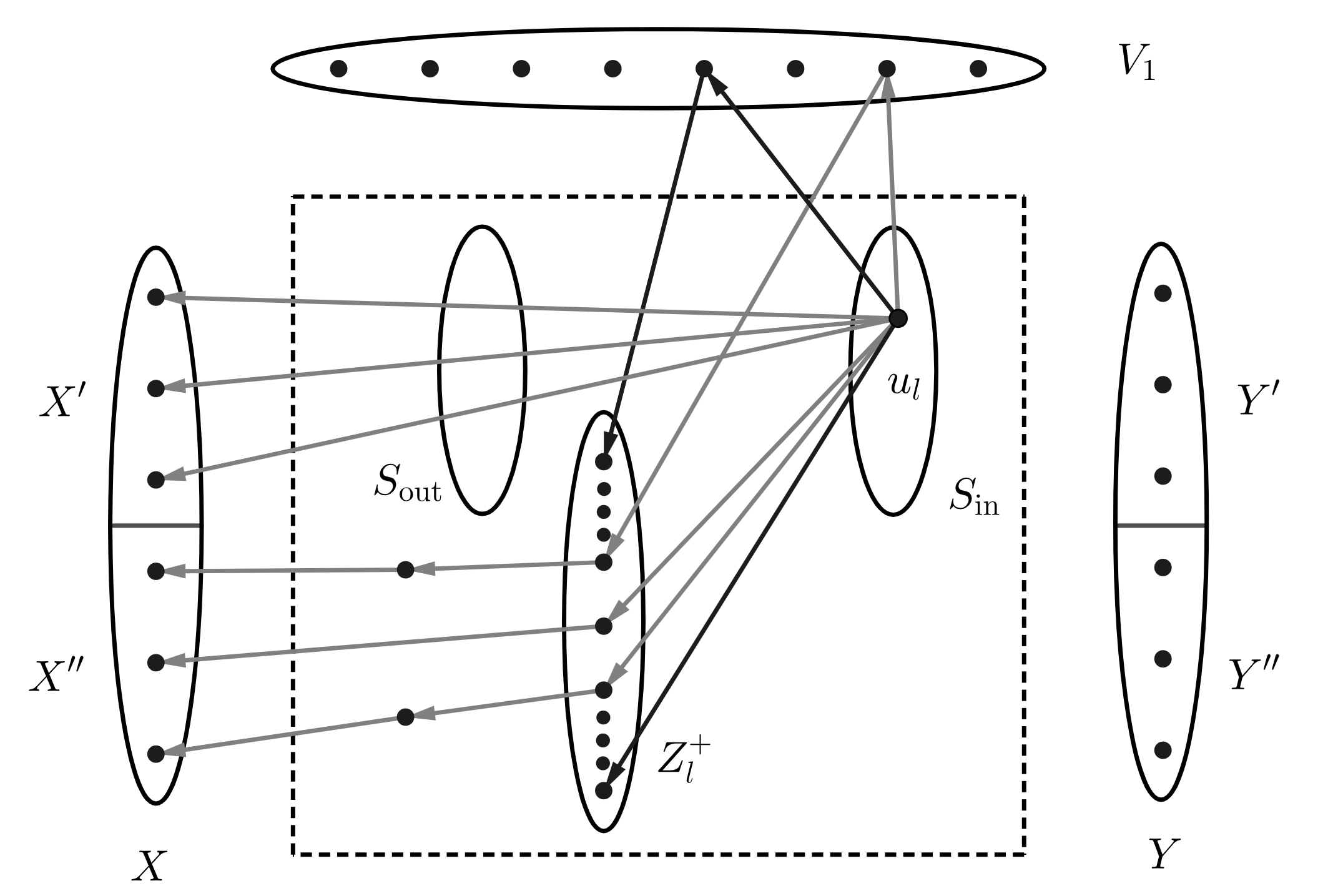}
    \caption{The construction of $\mathcal{Q}_{u_l,X}$: Black paths are redundant preliminary internally vertex-disjoint paths, and grey paths form the final selected fan $\mathcal{Q}_{u_l,X}$.}
    \label{fan:fig1}
\end{figure}


 {We have to show that} no arc from $\bigcup_{j\in [l-1]}\overleftarrow{A(\mathcal{Q}_{u_j,X})}$  $\cup \overleftarrow{A(\mathcal{Q}_{Y,u_j})}$ is used by $\mathcal{Q}'_l$. Assume for contradiction that there exists an arc $ e\in \overleftarrow{A(\mathcal{Q}_{u_j,X})}$  $\cup \overleftarrow{A(\mathcal{Q}_{Y,u_j})}$, for some $u_j\in S_{\mathrm{in}}\cup S_{\mathrm{out}}$,
belonging to $A(\mathcal{Q}'_l)(\subseteq A(T)\setminus \rev{A(\mathcal{P})})$. Hence $ \rev{e}\in (A(\mathcal{Q}_{u_j,X}\cup \mathcal{Q}_{Y,u_j}))\setminus A(T)$. {By} (V$_{l-1}$)-(VI$_{l-1}$), the arc $e$ is either from some vertex in $V_2\setminus (Y\cup X'')$ to $u_j\in S_{\mathrm{in}}$ or from $ u_j\in S_{\mathrm{out}}$ to some vertex in $V_2\setminus X$. Recall that $\mathcal{Q}'_l$ consists of disjoint paths  from $Z^+_l$ to $X''$ of each length at most 2. Accordingly, $u_j\in S_{\mathrm{in}}$ is an internal vertex of some path $Q\in \mathcal{Q}'_l$ in the former case. In the latter case, since the head of $e$ does not lie in $X$ and every path in $\mathcal{Q}'_l$ has length at most 2 and terminates in $X''\subseteq X$, the arc $e$ must be the first arc of its path; hence $u_j\in S_{\mathrm{out}}$ is its initial vertex. Neither possibility can occur, since all paths in $\mathcal{Q}'_l$ are internally disjoint from $S_{\mathrm{in}}$, and
\[
\mathrm{Init}(\mathcal{Q}'_l)\subseteq Z^+_l\subseteq V_2\setminus (X\cup Y\cup S_{\mathrm{out}}).
\]
Consequently, $\vQ_{u_l,X}$  {is the desired fan} in $ D_{l-1} $ satisfying (V$_{l}$). 



The construction of $\mathcal{Q}_{Y,u_l}=\{Q_{y_1,u_l},\dots,Q_{y_{3k+3},u_l}\}$ is symmetric if \(u_l\in S_{\mathrm{out}}\). Specifically, set $Q_{y_i,u_l}=y_iu_l$ whenever $y_iu_l\in A(D_{l-1})$. {Let $Y':=\{y_1,y_2,\dots,y_t\}$ be the set of those vertices in $Y$ which satisfy} $y_iu_l\in  A(D_{l-1})$, and let  $Y''=Y\setminus Y' $. We {first} find a $(V_2\setminus (X\cup Y\cup S_{\mathrm{in}}),u_l,{  13k+12-t},2)$-fan in $D_{l-1}$. By (I$_{l-1}$), the number of in-neighbours of $u_l$ in
$D_{l-1}$ which lie outside $X\cup Y\cup S_{\mathrm{in}}$
 is  at least
{ \[
\begin{aligned}
18k+14-|Y'|-|X\cup S_{\mathrm{in}}|
\geq 18k+14-t-(3k+3)-(2k-1)= 13k+12-t.
\end{aligned}
\]}
Choose ${  13k+12-t}$ such in-neighbors and denote them by $U^-$.
For each vertex $w\in U^-\cap V_1$, by \eqref{D0} and  (II$_{l-1}$), the number of in-neighbours of $w$ in $D_{l-1}$ {which lie} outside $X\cup Y\cup S_{\mathrm{in}}$ is at least
{ \begin{equation*}
    \begin{aligned}
        &\delta^0(D)-d^+_{D_0}(w)-(4k-3) -|X\cup Y\cup S_{\mathrm{in}}|\\
        &\geq 26k+15-1-(4k-3)-2(3k+3)-(2k-1)=14k+12> {  13k+12-t}.
    \end{aligned}
\end{equation*}}
Hence, the desired $(V_2\setminus (X\cup Y\cup S_{\mathrm{in}}),u_l,{  13k+12-t},2)$-fan {which we denote by $F^-_{u_l}$} can be found in $D_{l-1}$. Let $Z^-_l$ be {the} ${  13k+12-t} $ initial vertices of this fan. Applying Lemma~\ref{nearly} with $T:=T-X$, $R:=\mathcal{P}$, $W:=S_{\mathrm{out}}\cup Y'$ and $l_1:=1$, there exist $3k+3-t $ disjoint paths $\mathcal{Q}''_l$ from $Y''$ to $ Z^-_l $ of each length at most 2 in $(T-X)\setminus\rev{A(\mathcal P)}$ ,
which are internally disjoint with $S_{\mathrm{out}}\cup Y'$ (as $Y$ is nearly out-dominating in $T-X$ by the construction of $Y$). As above, this is valid because $Z_l^-\subseteq V(T)\setminus Y$, and the required cardinality inequality follows from $|Y''|=3k+3-t$. Similarly, no arc from $\bigcup_{j\in [l-1]}\overleftarrow{A(\mathcal{Q}_{u_j,X}\cup \mathcal{Q}_{Y,u_j})}$ is used by $\mathcal{Q}''_l$, since all paths in $\mathcal{Q}''_l$ are internally disjoint from $S_{\mathrm{out}}$, and $\mathrm{Ter}(\mathcal{Q}''_l)\subseteq Z^-_l\subseteq V_2\setminus (X\cup Y\cup S_{\mathrm{in}}).$ Together with the $(Z^-_l,u_l,{  13k+12-t},2)$-fan, a $(Y, u_l,3k+3,4)$-fan $\vQ_{Y,u_l}$ exists in $D_{l-1}$ (see Fig. \ref{fan:fig2}).

\begin{figure}[htbp]
    \centering
\includegraphics[width=0.6\textwidth]{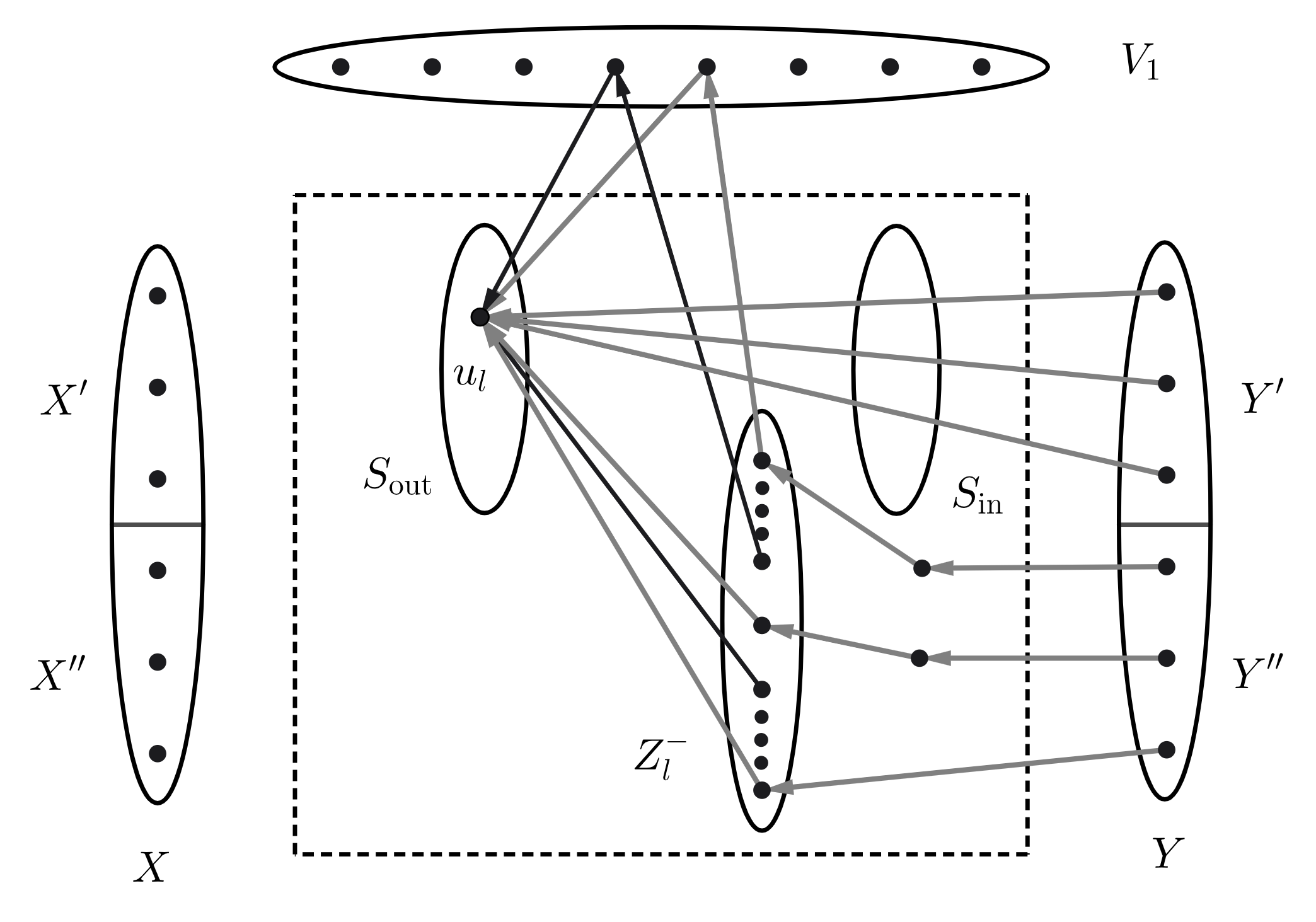}
    \caption{The construction of $\mathcal{Q}_{Y,u_l}$: Black paths are redundant preliminary internally vertex-disjoint paths, and grey paths form the final selected fan $\mathcal{Q}_{Y,u_l}$.}
    \label{fan:fig2}
\end{figure}


We now verify that $\mathcal{Q}_{u_l,X}$ and $\mathcal{Q}_{Y,u_l}$ satisfy conditions $(\text{I}_l)$–$(\text{VI}_l)$. Conditions {(III$_{l}$)}-(VI$_{l}$) hold trivially by construction. Condition $(\text{II}_l)$ follows from the fact that each individual fan $\mathcal{Q}_{u_l,X}$ uses at most one out-arc and one in-arc incident to any vertex $v\in V(D)\setminus \{u_1,\dots,u_l\}$, because its paths are internally disjoint. Since at least one fan corresponding to an unprocessed center is still absent, the number of already constructed individual fans is at most $|S_{\mathrm{in}}|+|S_{\mathrm{out}}|-1\le4k-3$. Condition $(\text{I}_l)$ can be derived from $(\text{II}_l)$, \eqref{D0} and the inequality
{ \[
\delta^0(D)\stackrel{(\text{II$_l$})}{-}(4k-3)\stackrel{\eqref{D0}}{-}(4k+1)\geq 26k+15-8k+2=18k+17.
\]}
The arguments above imply that we can  continue the construction process  until all $m$ centers have been processed.

{Let} $\mathcal{K}:=\bigcup_{v\in S_{\mathrm{in}}}\vQ_{v,X}\cup \bigcup_{v\in S_{\mathrm{out}}}\vQ_{Y,v}$.
\begin{claim}\label{claimk}
    The subdigraph $ \mathcal{K}$ satisfies the following.
    \begin{itemize}
        \item[$(K1)$] For every $v\in S_{\mathrm{out}}$, there is a $(Y,v,3k+3,4)$-fan using only arcs in $A(\mathcal{K})$. Also, for every $v\in S_{\mathrm{in}}$, there is a $(v,X,3k+3,4)$-fan using only arcs in $A(\mathcal{K})$.
    \item[$(K2)$] The set $A(D^0)\cup A(\mathcal{K})$ contains no $2$-cycle, and $|A(\mathcal{K})|\leq 48k^2+24k-24$.
    \item[$(K3)$] For each $v\in V_1\cup X\cup Y$, less than $4k$ in-arcs and $4k$ out-arcs of $v$ are in $A(\mathcal{K})\cup A(\mathcal{P} )$.
    \end{itemize}
\end{claim}
\begin{proof}
Property (K1) is immediate from the definition of $\mathcal{K}$. For (K2), at each step the new fan is constructed after deleting the reverse arcs of $D^0$ and of all previously constructed fans, while (IV$_l$) guarantees that the two fans with common center $u_l$ contain no 2-cycle. Hence $A(D^0)\cup A(\mathcal{K})$ contains no 2-cycle. Each fan contains $3k+3$ paths of length at most $4$, and there are at most $|S_{\mathrm{in}}|+|S_{\mathrm{out}}|\le 4k-2$ individual fans. Thus we have,
\[
|A(\mathcal K)|\leq 4(3k+3)(4k-2)=48k^2+24k-24.
\] Finally, every vertex of $V_1\cup X\cup Y$ lies outside $S_{\mathrm{in}}\cup S_{\mathrm{out}}$. Each individual fan uses at most one in-arc and one out-arc at such a vertex, so $K$ contributes at most $|S_{\mathrm{in}}|+|S_{\mathrm{out}}|\le 4k-2$ arcs in either direction. By \eqref{e5}, $\mathcal P$ contributes at most one more in-arc and one more out-arc.  Hence the number in either direction is at most $(4k-2)+1< 4k$.  Thus (K3) follows.
\end{proof}

Let $D^1$ be the spanning subdigraph of $D$ with $A(D^1)=A(D^0)\cup A(\mathcal{K})$, and let $ D^2:= D\setminus \rev{A(D^1)}$. According to \eqref{D0} and ($K3$), we obtain that,  for each vertex $v\in V_1\cup X\cup Y$,
{  $$d^{+}_{D^2}(v), d^{-}_{D^2}(v)\geq \delta ^0(D)\stackrel{\eqref{D0}}{-}1\stackrel{(K3)}{-}(4k-1) \geq 26k+15-4k =22k+15.$$}
After excluding the at most $|X\cup Y|=6k+6$ vertices of $X\cup Y$, there remain at least
{ \[
22k+15-(6k+6)=16k+9\]}
in-neighbours and out-neighbours outside $X\cup Y$. Since $A(D^1)\subseteq A(D^2)$, we must further exclude the arcs already contained in $D^1$. By (K3), there are at most $4k-1$ such arcs in either direction at $v$. Thus at least $16k+9-(4k-1)=12k+10$ arcs remain available in either direction in $A(D^2)\setminus A(D^1)$. 
Thus, for each $v\in V_1$, we choose
$\max\{k-d^+_{D^1}(v,V_2\setminus(X\cup Y)),0\}$ out-arcs $\mathcal O_v$ from $v$ to $V_2\setminus(X\cup Y)$
and
$\max\{k-d^-_{D^1}(v,V_2\setminus(X\cup Y)),0\}$ in-arcs $\mathcal I_v$ from $V_2\setminus(X\cup Y)$ to $v$, all from
$A(D^2)\setminus A(D^1)$. For each $v\in X\cup Y$, we choose $k$ out-arcs $\mathcal{O}_v$
from $v$ to $V(D)\setminus (X\cup Y)$ and $k$ in-arcs
$\mathcal{I}_v$ from $V(D)\setminus (X\cup Y)$ to $v$, again from
$A(D^2)\setminus A(D^1)$.
In each case, we choose these arcs so that
$\mathcal{O}_v\cup\mathcal{I}_v$ contains no $2$-cycle.
 These choices can be made greedily: first choose $\mathcal{O}_v$, and then choose $\mathcal{I}_v$ while avoiding the reverse arcs of $\mathcal{O}_v$. Since $|\mathcal O_v|\leq k$, at least
{ \[
12k+10-k=11k+10\geq k
\]}
candidates remain for $\mathcal I_v$. Thus $\mathcal O_v\cup\mathcal I_v$ contains no $2$-cycle.

We now construct the desired spanning oriented subdigraph $D^*$ as follows:
let $D^*$ be the spanning subdigraph of $D$ with arc set
\[
A(D^*):=A(D^1)\cup {[\bigcup_{v\in V_1\cup X\cup Y}(\mathcal{O}_v\cup \mathcal{I}_v)]}\cup A(T_{\mathrm{rev}}[X])\cup A(T_{\mathrm{rev}}[Y]).
\]
By (K2), the choice of $\mathcal O_v$ and $\mathcal I_v$, the definition
of $D^2$, and the fact that $T_{\rm rev}$ agrees with the arcs of
$\mathcal P$ inside $V_2$, the digraph $D^*$ is {an oriented graph}.  The construction is polynomial‑time. The sets $X,Y$ are obtained via the polynomial‑time test described above; the linkage $\mathcal P$ is computed by a standard vertex‑splitting max‑flow procedure and shortened greedily; Lemma~\ref{good graph} invokes polynomial‑time ordering and bipartite‑matching algorithms; Lemma~\ref{nearly} and the fan construction are constructive; and the remaining unused arcs are chosen greedily. Consequently, $D^*$ is constructible in polynomial time.

\begin{claim}

    $|A(D^*)|\le kn+k|V_1|+98k^2+38k+3.$
\end{claim}
\begin{proof}
Put $r:=|V(\mathcal P)\cap V_1|$. Since every vertex of $V(\mathcal P)\cap V_1$ is an internal vertex of a path in $\mathcal P$, exactly $2r$ arcs of $\mathcal P$ are incident with these vertices. At most $2|X\cup Y|=12k+12$ of the $2r$ arcs of $\mathcal P$ incident with $V(\mathcal P)\cap V_1$ have their other end in $X\cup Y$. Hence at least $2r-(12k+12)$ already join $V_1$ to $V_2\setminus (X\cup Y)$, and the choice of $\mathcal O_v$ and $\mathcal I_v$ give $|\bigcup_{v\in V_1}{(\mathcal  O_v\cup \mathcal I_v)}|\le 2k|V_1|-2r+12k+12.$ Since the $k$ paths in $\mathcal P$ are pairwise disjoint, $|A(\mathcal P)|=|V(\mathcal P)|-k$. Moreover, every vertex of $V(\mathcal P)\cap V_2$ belongs to $V(T'_{\mathrm{sparse}})\cup X\cup Y$. Therefore
\begin{equation*}
|A(\mathcal P)|\le |T'_{\mathrm{sparse}}|+r+|X\cup Y|-k=|T'_{\mathrm{sparse}}|+r+5k+6.
\end{equation*}
Using \eqref{degree1} and (K2), we obtain
\begin{equation*}
\begin{aligned}
|E_1|&:=\left|\bigcup_{v\in V_1}(\mathcal O_v\cup\mathcal I_v)\cup A(T'_{\mathrm{sparse}})\cup A(\mathcal P)\cup A(\mathcal K)\right|\\
&\leq \bigl(2k|V_1|-2r+12k+12\bigr)
 +\bigl((k-1)|T'_{\mathrm{sparse}}|+9k^2-12k+3\bigr)\\
&\quad +\bigl(|T'_{\mathrm{sparse}}|+r+5k+6\bigr)
 +(48k^2+24k-24)\\
&=2k|V_1|+k|T'_{\mathrm{sparse}}|-r+57k^2+29k-3\\
&\leq 2k|V_1|+k|T'_{\mathrm{sparse}}|+57k^2+29k-3.
\end{aligned}
\end{equation*}
Let $E_2:=\bigcup_{v\in X\cup Y}{(\mathcal O_v\cup \mathcal I_v)}$. Then $|E_2|\le 2k(6k+6)=12k^2+12k$. Also, since $T_{\mathrm{rev}}[X]$ and $T_{\mathrm{rev}}[Y]$ are tournaments of order $3k+3$, $|A(T_{\mathrm{rev}}[X])|+|A(T_{\mathrm{rev}}[Y])|=2\binom{3k+3}{2}=9k^2+15k+6.$
Hence, by \eqref{degree2}, \eqref{degree3}, and the definition of $D^*$,
\begin{equation*}
\begin{aligned}
|A(D^*)|
&\leq \bigl(2k|V_1|+k|T'_{\mathrm{sparse}}|+57k^2+29k-3\bigr)
 +(12k^2+12k)\\
&\quad +(9k^2+15k+6)
 +\bigl(k|T''_{\mathrm{sparse}}|+9k^2-4k\bigr)
 +(17k^2-8k)\\
&=2k|V_1|+k|T'_{\mathrm{sparse}}|+k|T''_{\mathrm{sparse}}|+104k^2+44k+3\\
&= 2k|V_1|+k|V_2\setminus(X\cup Y)|+104k^2+44k+3\\
&= 2k|V_1|+k|V_2|+98k^2+38k+3\\
&= kn+k|V_1|+98k^2+38k+3.
\end{aligned}
\end{equation*}
This proves the claim.
\end{proof}

Now, it is sufficient to prove that $D^*$ is $k$-strong. By the definition of $k$-strong, it is sufficient to prove that there is a $(u,v)$-path in $D^*\setminus S$, for {every}  set $S$ with $|S|\leq k-1$ and $u,v\in V(D^*)\setminus S$. From now on, fix any such set $S$, and two vertices $u,v\in V(D^*)\setminus S$. Our aim is to find a $(u,v)$-path in $D^*\setminus S$. Since $|S|\leq k-1$, there is a path ${P_r}\in \mathcal{P}$ such that ${P_r}\subseteq  D^*\setminus S$. 

\begin{claim}
\label{cl:insidegood}
    If $u,v\in V_2\setminus (X\cup Y)$, then $D^*\setminus S$ contains a $(u,v)$-path.
\end{claim}
\begin{proof}
If $u\in T''_{\text{sparse}}$, it follows from \eqref{e9} that either $u\in S''_{\mathrm{in}}$ or there exists a path in $T''_{\text{sparse}}\setminus S$ from $u$ to some vertex $u'\in S''_{\mathrm{in}}$. Combining with \eqref{e11}, either $u'\in S_{\mathrm{in}}$ or there exists a path in $T^{\text{abs}}_{\text{sparse}}\setminus S$ from $u'$ to some vertex $u''\in S_{\mathrm{in}}$. Consequently, a path from $u$ to $u''\in S_{\mathrm{in}}$ exists within $D^*\setminus S$.

If $u\in T'_{\text{sparse}}$ and $u\in S'_{\mathrm{in}}$ or $u$ admits a path to $S'_{\mathrm{in}}$ in $T'_{\text{sparse}}\setminus S$, then \eqref{e11} guarantees the existence of a path from $u$ to a vertex $u''\in S_{\mathrm{in}}$ in $D^*\setminus S$. Otherwise, $u\notin S'_{\mathrm{in}}$ and no such path exists in $T'_{\text{sparse}}\setminus S$. Recall that $\sigma '_{\mathrm{sparse}}$ is an order  of $V(D'_{\mathrm{rev}})$ such that $T'_{\mathrm{sparse}}$ is $(\sigma '_{\mathrm{sparse}},k-1,2)$-nice and  $S'_{\mathrm{in}}:=\sigma '_{\mathrm{sparse}}(|D'_{\mathrm{rev}}|-2k+ 2, |D'_{\mathrm{rev}}|) $.  Denote $i$ be the maximum index such that $u$ can reach $v_i$ in the order $\sigma '_{\mathrm{sparse}}$ by a path in $T'_{\text{sparse}} - S$. If $i \geq |D'_{\mathrm{rev}}| - 2k+2$, then $v_i \in S'_{\mathrm{in}}$, a contradiction. Consequently, $i < |D'_{\mathrm{rev}}| - 2k+2$.  Since every vertex in $\sigma '_{\mathrm{sparse}}(1,|D'_{\mathrm{rev}}|-2k+1)$ has out-degree at least $k-1$ in $T'_{\mathrm{sparse}}$ (see (D3) in Definition \ref{def3}) and $|S| \leq k-1$, we must have  $S = N_{T'_{\mathrm{sparse}}}^+(v_i)$ and $|S|=k-1$. The definition of $D'_{\mathrm{rev}}$ implies that $v_i$ lies on some minimal path $P_j\subseteq \mathcal{P}$ from $x_j$ to $y_j$. Let $Q'$ denote the subpath of $P_j$ starting at $v_i$ and ending at $y_j$, and let $w$ be the out-neighbor of $v_i$ on $Q'$. As $P_j$ is minimal in $D$, $V(Q')\cap N_D^+(v_i)=\{w\}$. Together with $T'_{\text{sparse}}\subseteq T_{\mathrm{rev}}[V_2^2]\setminus A(\mathcal{P})$, we further have $V(Q')\cap N_{T'_{\text{sparse}}}^+(v_i)=V(Q')\cap S=\emptyset$. Let $Q''$ be a path in $T'_{\text{sparse}}\setminus S$ from $u$ to $v_i$ {(such a path exists by the definition of $v_i$ above)}. Then $Q''\circ Q'$ is a directed walk in $D^*\setminus S$ from $u$ to $y_j$.

As $d^+_{T[Y]}(y_j)\geq k+1$, it follows that  $d^+_{T_{\mathrm{rev}}[Y]}(y_j)\geq k$. Choose a vertex $y_j''\in S_{\mathrm{out}}\setminus S$; such a vertex exists by (\ref{e11}). Consider the $(Y,y_j'',3k+3,4)$-fan contained in $\mathcal{K}$. Among the at least $k$ vertices in $N^+_{T_{\mathrm{rev}}[Y]}(y_j)$, at most $|S|\le k-1$ {of the paths in the above fan intersect} $S$, because the paths of this fan are internally disjoint and have distinct initial vertices. Hence there is a vertex $y_j'\in N^+_{T_{\mathrm{rev}}[Y]}(y_j)$ such that the arc $y_jy_j'$ and the corresponding $(y_j',y_j'')$-path $P'$ both lie in $D^*\setminus S$. By (\ref{e11}), there is a path $P''$ in $T^{abs}_{\text{sparse}}\setminus S$ from $y_j''$ to a vertex of $S_{\mathrm{in}}\setminus S$. Therefore $Q''\circ Q'\circ y_jy_j'\circ P'\circ P''$ is a directed walk from $u$ to a vertex $u''\in S_{\mathrm{in}}$ in $D^*\setminus S$, {implying that $D^*\setminus S$ contains a $(u,u'')$-path}


{Analogous arguments show that for every} vertex $v\in V(T'_{\text{sparse}})\cup V(T''_{\text{sparse}})$, a path $R''$ from some vertex $v''\in S_{\mathrm{out}}$ to $v$ can be found in $D^*\setminus S$. Combined with condition ($K1$), $d^-_{T_{\mathrm{rev}}[X]}(x_{{r}})\geq k$ and $d^+_{T_{\mathrm{rev}}[Y]}(y_{{r}})\geq k$, there exists a path from $y_{{r}}$ to $v''$ and a path from $u''$ to $x_{{r}}$ within $D^*\setminus S$. Since ${P_r}\subseteq D^*\setminus S$, it follows that the desired $(u,v)$-path exists in $D^*\setminus S$ {and the claim is proved}.
\end{proof}

{It follows from Claim \ref{cl:insidegood} that to complete the proof that $D^*$ is $k$-strong it} is sufficient to verify that, for each vertex $v\in V_1\cup X\cup Y$, $D^*\setminus S$ contains
\begin{equation}\label{strong-connected}
\text{a path from $v$ to $V_2\setminus (X\cup Y)$ and a path from $V_2\setminus (X\cup Y)$ to $v$.}
\end{equation}
First consider $v\in V_1$. By the choice of $\mathcal O_v$ and
$\mathcal I_v$, we obtain
$d^+_{D^*}(v,V_2\setminus(X\cup Y))\ge k$
 and  $d^-_{D^*}(v,V_2\setminus(X\cup Y))\ge k$. Since $|S|\le k-1$, at least one out-neighbour and at least one
in-neighbour of $v$ in $V_2\setminus(X\cup Y)$ remain outside $S$.
Hence $D^*\setminus S$ contains both a path from $v$ to $V_2\setminus(X\cup Y)$ and a path from $V_2\setminus(X\cup Y)$ to $v$. Now let $v\in X\cup Y$. By construction, $\mathcal O_v$ consists
of $k$ out-arcs from $v$ to distinct vertices of
$V(D)\setminus(X\cup Y)$. Hence some arc $vw\in\mathcal O_v$
satisfies $w\notin S$. If $w\in V_2\setminus(X\cup Y)$, we are done.
Otherwise, $w\in V_1$, and the preceding paragraph gives a path
from $w$ to $V_2\setminus(X\cup Y)$ in $D^*\setminus S$. Thus
$v$ has a path to $V_2\setminus(X\cup Y)$ in $D^*\setminus S$.
The reverse argument, using $\mathcal I_v$, gives a path from
$V_2\setminus(X\cup Y)$ to $v$. Therefore \eqref{strong-connected} holds, and hence $D^*$ is $k$-strong.
\end{proof}

\section{Remarks  {and open problems}}\label{sec:remarks}
 {The construction below shows that one cannot replace the  minimum semi-degree assumption in Theorem~\ref{main1} by a lower bound on the out-degrees}.

\begin{proposition}
  Let $k\geq 2$ be a positive integer.  There exists {an infinite} family of $(2k-3)$-strong split digraphs $D$ of order $n$ with $\delta^+(D)\geq n/2-1$ that contains no spanning $k$-strong oriented subdigraph.
\end{proposition}
\begin{proof}
    We define a family of split digraphs $D=(V_1,V_2;A)$ of order $n$ as follows. Let $D[V_1]$ be an arcless digraph of order $\lfloor n/2\rfloor-2k+3$. Let $D[V^1_2]$ be a complete digraph on $2k-3$ vertices, and let $D[V^2_2]$ be a complete digraph on $\lceil n/2\rceil$ vertices. All arcs between $V^1_2$ and $V^2_2$ in $D$ are directed from $V^1_2$ to $V^2_2$. The bipartite digraph $D[V_1,V^2_2]$ is complete, and there are exactly $2k-3$ disjoint $2$-cycles between $V_1$ and $V^1_2$. By construction, it is straightforward to verify that $D$ satisfies $\delta^+(D)\geq \lceil n/2 \rceil-1$, and $D$ contains no spanning $k$-strong oriented subdigraph. Indeed, if we delete one arc from each $2$-cycle in $D$, some vertex of $D[V^1_2]$ will have in-degree at most $k-1$. 


  We prove that $D$ is $(2k-3)$-strong. Assume, to the contrary, that $D$ is not $(2k-3)$-strong. Then there exists a vertex set $S\subseteq V(D)$ with $|S|\leq 2k-4$ and two vertices $u,v\in V(D)\setminus S$ such that $D-S$ contains no $(u,v)$-path. Our construction implies that an arc $uv$ or a 2-path $uwv$ always exists in $D-S$, which eliminates the cases $u,v\in V_1\cup V_2^2$, $u,v\in V_2^1$, and $u\in V_2^1, v\in V_2^2$. It therefore suffices to verify the remaining three scenarios: $u\in V^2_2, v\in V^1_2$; $ u\in V^1_2, v\in V_1 $ and $ u\in V_1, v\in V_2^1 $. By construction, $D[V_1,V_2^2]$ is a complete bipartite digraph, and there are $2k-3$ disjoint $2$-cycles between $V_1$ and $V_2^1$. Since $|S|\leq 2k-4$, at least one such $2$-cycle remains in $D-S$. This ensures the existence of a $2$-path $P$ from a vertex $u'\in V^2_2$ to a vertex in $u''\in V^1_2$ and a 2-path $Q$ from a vertex $v'\in V^1_2$ to a vertex in $v''\in V^2_2 $.

For $u\in V_2^2, v\in V_2^1$, the completeness of $D[V^1_2]$ and $D[V^2_2]$ yields $uu',u''v\in A(D-S)$. Then $uu'\circ P\circ u''v$ defines a $(u,v)$-path in $D-S$. For $u\in V_2^1, v\in V_1$, we have $uv',v''v\in A(D-S)$, so $uv'\circ Q\circ v''v$ is a $(u,v)$-path in $D-S$. For $u\in V_1, v\in V_2^1$, it holds that $uu',u''v\in A(D-S)$, and hence $uu'\circ P\circ u''v$ forms a $(u,v)$-path in $D-S$. All configurations contradict the initial assumption. Consequently, $D$ is $2k-3$-strong. This completes the proof.
\end{proof}

Bang-Jensen and Jordán \cite{bangDM310} constructed an infinite family of $(2k-2)$-strong semicomplete digraphs admitting no spanning $k$-strong tournament. This demonstrates that the minimum semi-degree bound in Theorem~\ref{main1}  cannot be improved below $2k-1$.

A result established by Kang shows that even sufficiently large minimum semi-degree cannot reduce the sparsity of any spanning $k$-strong subdigraph in some split digraphs (see the construction of digraph $G_{n_1,n_2,k,\overline{\Delta}}$ in \cite{kang2018}).
\begin{proposition}\label{prop2}\cite{kang2018}
    There exists a family of $k$-strong split digraphs $D$ of order $n$ with $\delta^0(D)\geq \lfloor \tfrac{n-|V_1|}{4}\big\rfloor-1$ such that every spanning $k$-strong subdigraph $D'$ of $D$ has at least $kn+k|V_1|$ arcs.
\end{proposition}

 Recall that a result of Bang-Jensen, Huang, and Yeo \cite{bangJGT46} implies the existence of a family of $k$-strong tournaments of order $n$ {with the property that} every spanning $k$-strong subdigraph of {such a tournament} must contain at least $kn + k(k-1)/2$ arcs. Combined with Proposition \ref{prop2}, this shows that the arc bound $kn+k|V_1|+98k^2+38k+3$ in Theorem \ref{main1} is optimal in a certain sense, {namely, even for tournaments we cannot get below $kn$ plus a quadratic function of $k$.}

\begin{problem}\label{prob1}
Determine the optimal minimum semi-degree condition in Theorem~\ref{main1}. 
\end{problem}

A digraph $D=(V,A)$ is {\bf $k$}-arc-strong if $D\setminus A'$ is strong for every subset $A'\subset A$ of size at most $k-1$.
\begin{theorem}[Nash-Williams]
\cite{nashwilliamsCJM12}\label{thm:nwthm} A graph $G$ has a $k$-arc-strong orientation if and only if $G$ is $2k$-edge-connected.
\end{theorem}

Frank \cite{frank1995} conjectured that a graph $G=(V,E)$ has a $k$-strong orientation if and only if $G[V\setminus X]$ is $2(k-|X|)$-edge-connected for every subset $X\subset V$ of size at most $k-1$. A graph $G$ satisfying this condition is called {\bf weakly-$2k$-connected}. The condition is clearly necessary by Theorem \ref{thm:nwthm}. Thomassen proved Frank's conjecture  for $k=2$ \cite{186} but Durand de Gevigny disproved the conjecture for all higher values of $k$ \cite{degevignyJCT141}. The counterexamples have several vertices of degree in $2k$. Inspired by our results which show that for $k$-strong split digraphs, putting a lower bound on the minimum semi-degree is enough to guarantee a spanning $k$-strong oriented subdigraph, we pose the following problem.

\begin{problem}
    Does there exist a function $h(k)$ such that every graph $G$ which is weakly-$2k$-connected and with minimum degree $\delta{}(G)\geq h(k)$ has a $k$-strong orientation?
\end{problem}

By the result of Garamvolgyi et al. \cite{garamvolgyi2025highly} we know that every $320k^2$-connected graph has a $k$-strong orientation so the following
is a natural step on the way to proving Conjecture \ref{conj1} in the case of symmetric digraphs.

\begin{conjecture}

    There exists a function $z(k)$ such that every $2k$-strong symmetric digraph $D$ with $\delta^0(D)\geq z(k)$ has a spanning $k$-strong oriented graph?

\end{conjecture}

\bibliography{refs1} 

\begin{thebibliography}{10}

\bibitem{bang2009}
J.~Bang-Jensen.
\newblock Problems and conjectures concerning connectivity, paths, trees and cycles in tournament-like digraphs.
\newblock {\em Discrete Math.}, 309:5655--5667, 2009.

\bibitem{book}
J.~Bang-Jensen and G.~Gutin.
\newblock {\em Digraphs: Theory, Algorithms and Applications}.
\newblock Springer-Verlag London, 2009.

\bibitem{13}
J.~Bang-Jensen and J.~Huang.
\newblock Quasi-transitive digraphs.
\newblock {\em J. Graph Theory}, 20(2):141--161, 1995.

\bibitem{bangJGT46}
J.~Bang-Jensen, J.~Huang, and A.~Yeo.
\newblock {Spanning {$k$}-arc-strong subdigraphs with few arcs in {$k$}-arc-strong tournaments}.
\newblock {\em J. Graph Theory}, 46(4):265--284, 2004.

\bibitem{bangDM310}
J.~Bang{-}Jensen and T.~Jord{\'{a}}n.
\newblock Spanning 2-strong tournaments in 3-strong semicomplete digraphs.
\newblock {\em Discret. Math.}, 310(9):1424--1428, 2010.

\bibitem{BangJensenWangSplit}
J.~Bang-Jensen and Y.~Wang.
\newblock Strong arc decompositions of split digraphs.
\newblock {\em J. Graph Theory}, 108(1):5--26, 2025.

\bibitem{berg2005}
A.~R. Berg and T.~Jord{\'a}n.
\newblock Minimally $k$-edge-connected directed graphs of maximum size.
\newblock {\em Graphs Combin.}, 21:39--50, 2005.

\bibitem{boesch1980}
F.~Boesch and R.~Tindell.
\newblock Robbins's theorem for mixed multigraphs.
\newblock {\em Amer. Math. Mon.}, 87(9):716--719, 1980.

\bibitem{ChenBangJensenYanZhou2026}
X.~Chen, J.~Bang-Jensen, J.~Yan, and J.~Zhou.
\newblock On the $2$-linkage problem for split digraphs.
\newblock arXiv:2603.07603.

\bibitem{dalmazzo1977}
M.~Dalmazzo.
\newblock Nombre d'arcs dans les graphes $k$-arc-connexes minimaux.
\newblock {\em Acad. Sci. Paris Sér. A-B}, 285:A341--A344, 1977.

\bibitem{degevignyJCT141}
O.~Durand de~Gevigney.
\newblock {On Frank's conjecture on \emph{k}-connected orientations}.
\newblock {\em J. Comb. Theory {B}}, 141:105--114, 2020.

\bibitem{frank1995}
A.~Frank.
\newblock Connectivity and network flows.
\newblock In {\em Handbook of combinatorics, Vol. 1}, pages 111--177. Elsevier, 1995.

\bibitem{garamvolgyi2025highly}
D{\'a}niel Garamv{\"o}lgyi, Tibor Jord{\'a}n, Csaba Kir{\'a}ly, and Soma Vill{\'a}nyi.
\newblock Highly connected orientations from edge-disjoint rigid subgraphs.
\newblock {\em Forum Math. Pi}, 13:e11, 2025.

\bibitem{5}
Y.~Guo.
\newblock Spanning local tournaments in locally semicomplete digraphs.
\newblock {\em Discrete Appl. Math.}, 79:119--125, 1997.

\bibitem{115}
T.~Jordán.
\newblock On the existence of $k$ edge-disjoint 2-connected spanning subgraphs.
\newblock {\em J. Combin. Theory Ser. B}, 95(2):257--262, 2005.

\bibitem{kang2021}
D.~Kang and J.~Kim.
\newblock On 1-factors with prescribed lengths in tournaments.
\newblock {\em J. Comb. Theory B}, 141:31--71, 2020.

\bibitem{kang2017}
D.~Y. Kang, J.~Kim, Y.~Kim, and G.~Suh.
\newblock Sparse spanning $k$-connected subgraphs in tournaments.
\newblock {\em SIAM J. Discrete Math.}, 31:2206--2227, 2017.

\bibitem{kang2018}
D.Y. Kang.
\newblock Sparse highly connected spanning subgraphs in dense directed graphs.
\newblock {\em Combin. Probab. Comput.}, 27:892--907, 2018.

\bibitem{mader1985}
W.~Mader.
\newblock Minimal $n$-fach zusammenhängende digraphen.
\newblock {\em J. Combin. Theory Ser. B}, 38:102--117, 1985.

\bibitem{menger}
K.~Menger.
\newblock Zur allgemeinen kurventheorie.
\newblock {\em Fund. Math.}, 10:96--115, 1927.

\bibitem{nashwilliamsCJM12}
C.St.J.A. Nash-Williams.
\newblock On orientations, connectivity and odd-vertex- pairings in finite graphs.
\newblock {\em Can. J. Math}, 12:555--567, 1960.

\bibitem{NguyenScottSeymour}
Tung Nguyen, Alex Scott, and Paul Seymour.
\newblock Distant digraph domination.
\newblock {\em Electron. J. Combin.}, 33(1):P1.32, 2026.

\bibitem{Robbins1939}
H.E. Robbins.
\newblock A theorem on graphs with an application to a problem on traffic control.
\newblock {\em Amer. Math. Mon.}, 46:281--283, 1939.

\bibitem{61}
Y.~Sun and Z.~Jin.
\newblock Semicomplete compositions of digraphs.
\newblock {\em Discrete Math.}, 346(8):113420, 2023.

\bibitem{thomassen1989}
C.~Thomassen.
\newblock Configurations in graphs of large minimum degree, connectivity, or chromatic number.
\newblock {\em Ann. N. Y. Acad. Sci.}, 555:402--412, 1989.

\bibitem{186}
C.~Thomassen.
\newblock Strongly 2-connected orientations of graphs.
\newblock {\em J. Combin. Theory Ser. B}, 110:67--78, 2015.

\bibitem{62}
L.~Volkmann.
\newblock Cycles in multipartite tournaments: results and problems.
\newblock {\em Discrete Math.}, 245(1):19--53, 2002.

\bibitem{wang2024spanning}
K.~Wang, Y.~Qi, and J.~Yan.
\newblock Spanning 3-strong tournaments in 5-strong semicomplete digraphs.
\newblock {\em Discrete Math.}, 347(1):113714, 2024.

\bibitem{zhouDM26}
J.~Zhou, J.~Bang-Jensen, and J.~Yan.
\newblock On the $k$-linkage problem for generalizations of semicomplete digraphs.
\newblock {\em Discrete Math.}, 349(2):114700, 2026.

\bibitem{ZhouBangJensenZhouYan2026}
J.~Zhou, J.~Bang-Jensen, T.~Zhou, and J.~Yan.
\newblock Highly connected spanning oriented subdigraphs in generalizations of semicomplete digraphs, 2026.
\newblock arXiv:2607.17150.

\bibitem{zhou2025proof}
J.~Zhou and J.~Yan.
\newblock Proof of the linkage conjecture for highly connected tournaments, 2025.
\newblock arXiv:2507.22651.

\bibitem{Zhouarxiv}
T.~Zhou, J.~Bang-Jensen, J.~Zhou, and J.~Yan.
\newblock $k$-arc-strong orientations of semicomplete digraphs, 2026.
\newblock arXiv:2607.17116.

\end{thebibliography}

\end{document}